\documentclass[draft]{amsart}

\usepackage[pdfborder={0 0 0},draft=false]{hyperref}

\usepackage{esint,amssymb,bm}

\newcommand\abs[2][empty]{\csname#1\endcsname \lvert{#2}\csname#1\endcsname\rvert}
\newcommand\doublebar[2][empty]{\csname#1\endcsname \lVert{#2}\csname#1\endcsname\rVert}

\newcommand\dist{\mathop{\mathrm{dist}}\nolimits}
\newcommand\Div{\mathop{\mathrm{div}}\nolimits}

\newcommand\supp{\mathop{\mathrm{supp}}\nolimits}

\newcommand\re{\mathop{\mathrm{Re}}\nolimits}

\newcommand\R{\mathbb{R}} 
\newcommand\C{\mathbb{C}} 
 
\newcommand\N{\mathbb{N}}
\newcommand\1{\mathbf{1}}

\newcommand\eqitem{\refstepcounter{equation}%
	\item[\textup{(\theequation)}]%
	\expandafter\gdef\csname@currentlabel\endcsname{\theequation}}
\newcommand\mydot{\dot} 
\newcommand\mat{\bm} 
\newcommand\dmn{d}
\newcommand\pdmn{d}
\newcommand\dmnMinusOne{{d-1}}

\newcommand\pdmnMinusOne{{(d-1)}}

\newtheorem{thm}[equation]{Theorem}
\newtheorem{lem}[equation]{Lemma}
\newtheorem{cor}[equation]{Corollary}

\theoremstyle{definition}

\theoremstyle{remark}
\newtheorem{rmk}[equation]{Remark}

\numberwithin{equation}{section}
\numberwithin{figure}{section}



\begin{document}

\title[Gradient estimates and the fundamental solution]{Gradient estimates and the fundamental solution for higher-order elliptic systems with rough coefficients}

\author{Ariel Barton}
\address{Ariel Barton, 202 Math Sciences Bldg., University of Missouri, Columbia, MO 65211}
\email{bartonae@missouri.edu}

\begin{abstract}
We extend several well-known tools from the theory of second-order divergence-form elliptic equations to the case of higher-order equations. These tools are the Caccioppoli inequality, Meyers's reverse H\"older inequality for gradients, and the fundamental solution. Our construction of the fundamental solution may also be of interest in the theory of second-order operators, as we impose no regularity assumptions on our elliptic operator beyond ellipticity and boundedness of coefficients.
\end{abstract}

\maketitle


\section{Introduction}

In this paper we will study divergence-form elliptic operators $L$ of order~$2m$, given formally by
\begin{equation*}
(L\vec u)_j = (-1)^m\sum_{k=1}^N \sum_{\abs{\alpha}=m} \sum_{\abs{\beta}=m} \partial^\alpha (A_{\alpha\beta}^{jk} \partial^{\beta} u_k)
\end{equation*}
and in particular systems of equations of the form $(L \vec u)_j= (-1)^m\sum_{\abs{\alpha}=m} \partial^\alpha F_{j,\alpha}$. (We will write this system of equations as $L\vec u=\Div_m\mydot F$.)

The theory of second-order operators, that is, operators with $m=1$, 
has a long and celebrated history. 
Important tools in the theory of second-order elliptic systems include the Caccioppoli inequality, Meyers's reverse H\"older inequality for derivatives, and the fundamental solution. 

The boundary Caccioppoli inequality states that, if $L\vec u= \Div \mydot F$ in some domain~$\Omega$ for some second-order elliptic operator~$L$, and if either $\vec u=0$ or $\nu\cdot \mat A\nabla \vec u=0$ on $\partial\Omega\cap B(x_0,r)$, where $\nu$ is the unit outward normal vector, 
then the gradient of $\vec u$ may be controlled by $\vec u$ and the inhomogeneous term~$\mydot F$, as
\begin{equation}\label{eqn:Caccioppoli:2}
\int_{B(x_0,r)\cap\Omega}\abs{\nabla \vec u}^2\leq \frac{C}{r^2}
\int_{B(x_0,2r)\cap\Omega}\abs{\vec u}^2 
+ C\int_{B(x_0,2r)\cap\Omega}\abs{\mydot F}^2.\end{equation}

Meyers's reverse H\"older estimate (see \cite{Mey63}) states that, if $L \vec u=\Div \mydot F$ in some ball $B(x_0,r)$, then $\nabla \vec u$ satisfies the reverse H\"older estimate 
\begin{equation}\label{eqn:Meyers:2}
\biggl(\int_{B(x_0,r)}\abs{\nabla \vec u}^p\biggr)^{1/p}
\leq \frac{C}{r^{\pdmn/2-\pdmn/p}}
\biggl(\int_{B(x_0,2r)}\abs{\nabla \vec u}^2\biggr)^{1/2} 
+ C\biggl(\int_{B(x_0,2r)}\abs{\mydot F}^p\biggr)^{1/p}\end{equation}
for some $p>2$ depending only on the operator~$L$. With some care, Meyers's estimate may also be extended to the boundary case, at least in relatively nice domains. Both of these inequalities have been used extensively in the literature.

Much less is known in the case of higher-order elliptic systems in the rough setting.
In the case of continuous coefficients and $C^m$ domains, some regularity results are available; see \cite{AgmDN64}. In the interior case the Caccioppoli inequality 
\begin{equation}
\label{eqn:Caccioppoli:Cam80}
\int_{B(x_0,r)}\abs{\nabla^m \vec u}^2
\leq 
\sum_{j=0}^{m-1}\frac{C}{r^{2m-2j}}
\int_{B(x_0,2r)}\abs{\nabla^j\vec u}^2 
+
{C}
\int_{B(x_0,2r)}\abs{\mydot F}^2 
\end{equation}
was established in \cite{Cam80} for general bounded and strongly elliptic coefficients. It would of course be preferable to establish this bound with only a norm of $\vec u$, and not of~$\nabla^j\vec u$, on the right-hand side.
In \cite{AusQ00}, the authors established the bound
\begin{equation}
\label{eqn:Caccioppoli:AusQ00}
\int_{B(x_0,r)}\abs{\nabla^m \vec u}^2
\leq 
\frac{C(\varepsilon)}{r^2}
\int_{B(x_0,2r)}\abs{\vec u}^2 
+
\varepsilon
\int_{B(x_0,2r)}\abs{\nabla^m \vec u}^2 
\end{equation}
for solutions $\vec u$ to the equation $L\vec u=0$ in $B(x_0,2r)$, where $\varepsilon$ is an arbitrary positive number and $C(\varepsilon)$ a constant depending on $\varepsilon$. Either of the bounds \eqref{eqn:Caccioppoli:Cam80} or \eqref{eqn:Caccioppoli:AusQ00} suffices to generalize Meyers's estimate \eqref{eqn:Meyers:2} to the higher-order case, and in fact this is done in both \cite{Cam80} and~\cite{AusQ00}.

The boundary Caccioppoli inequality in the case of rough domains has not been established; we mention that some pointwise estimates were established in \cite{MayMaz08,MayMaz09B} in the case where $L=\Delta^2$ is the biharmonic operator.

In Section~\ref{sec:Caccioppoli}, we will establish the higher-order Caccioppoli inequality with no terms involving derivatives of~$\vec u$ on the right-hand side; we will also establish this inequality in the Dirichlet and Neumann boundary cases. The main results of this section are Lemma~\ref{lem:Caccioppoli:Neumann} and Corollaries~\ref{cor:Caccioppoli:interior} and~\ref{cor:Caccioppoli:Dirichlet}.
In Section~\ref{sec:Meyers}, we will provide boundary versions and some refinements to the generalization of Meyers's inequality \eqref{eqn:Meyers:2}, and in particular will carefully state the consequences for the lower-order derivatives of the solution~$\vec u$. The main results of this section are Theorems~\ref{thm:Meyers} and~\ref{thm:Meyers:Neumann}.

Another important tool in the second-order case is the fundamental solution  $\mat E^L(x,y)$. This solution is a (matrix-valued) distribution defined on $\R^\dmn\times\R^\dmn$ such that, formally, $L \mat E^L(\,\cdot\,,y)=\mat I\delta_y$, where $\delta_y$ denotes the Dirac mass and $\mat I$ denotes the identity matrix. 
In Section~\ref{sec:fundamental} we will construct the fundamental solution for higher-order elliptic systems.

The fundamental solution was constructed for second-order equations with real coefficients (that is, if $N=m=1$, $A_{\alpha\beta}$ real)  in \cite{LitSW63} (in the case of symmetric coefficients $A_{\alpha\beta}=A_{\beta\alpha}$), in
\cite{GruW82} (in dimension $\dmn\geq 3$) and in \cite{KenN85} (in dimension $\dmn=2$). In dimension $\dmn=2$ these results were extended to the case of complex coefficients in \cite{AusMT98}; as observed in \cite{DonK09} their strategy carries over to the case of systems with $\dmn=2$, $m=1$ and $N\geq 1$.

In the case of second-order systems (that is, $m=1$ and $N\geq 1$),
the fundamental solution was constructed in the papers \cite{Fuc86,DolM95,HofK07,Ros13} under progressively weaker conditions on the operator~$L$. 

Specifically, the paper \cite{Ros13} constructs the fundamental solution for the operator $L$ under the assumption that, if $L\vec u=0$ in some ball~$B(x,r)$, then $\vec u$ is continuous in~$B(x,r)$ and satisfies the local boundedness estimate
\begin{equation}\label{eqn:Moser}
\abs{\vec u(x)}\leq \biggl(\frac{1}{r^\dmn}\int_{B(x,r)} \abs{\vec u}^2\biggr)^{1/2}.\end{equation}
This assumption is not true for all elliptic operators; see \cite{Fre08}.

All of the above papers made the same or stronger assumptions.  Specifically, \cite{Fuc86,DolM95} constructed the fundamental solution in the case of systems with continuous coefficients, for which the bound  \eqref{eqn:Moser} is always valid; see \cite[Theorem~6.4.8]{Mor66} or \cite[Section~3]{DolM95}. \cite{HofK07} constructed the fundamental solution using the stronger assumption of local H\"older continuity of solutions.  The papers \cite{LitSW63,GruW82,KenN85} considered only the case $N=m=1$ with real coefficients; in this case the bound \eqref{eqn:Moser} was established by Moser in \cite{Mos61}.
The paper \cite{AusMT98} constructed the fundamental solution in dimension $\dmn=2$. In this case Meyers's estimate~\eqref{eqn:Meyers:2} implies that solutions $\vec u$ locally satisfy $\nabla \vec u\in L^p$ for some $p>\dmn$; Morrey's inequality then implies that solutions are necessarily locally H\"older continuous. The papers \cite{DonK09,KanK10,ChoDK12} investigate the related topic of Green's functions in domains; they too require local boundedness of solutions (either as an explicit assumption or by virtue of working in dimension $\dmn=2$).

Fewer results are available in the case of higher-order equations. In the case of the polyharmonic operator $L=(-\Delta)^m$ we have an explicit formula for the fundamental solution, and this solution has been used extensively in the theory of biharmonic and polyharmonic functions. 
The fundamental solution in the case of general constant coefficients has also been used; see, for example,  \cite{Fri61,PipV95b,Ver96,Maz02,MazMS10,Dal13,DalMM13}. 
In the case of variable analytic coefficients the fundamental solution was constructed in \cite{Joh55}, and in the case of smooth coefficients the Green's function in domains was constructed in \cite{Dud01}.

We will initially construct the fundamental solution for higher-order systems only in the case where solutions are continuous and satisfy the local bound~\eqref{eqn:Moser}. Again by Morrey's inequality and the higher-order generalizations of the Caccioppoli inequality~\eqref{eqn:Caccioppoli:2}, this is true whenever the elliptic operator $L$ is of order $2m>\dmn$. Thus, we will begin by constructing the fundamental solution in the case of low dimension or high order. Then, given an operator $L$ of order $2m\leq\dmn$, we will construct an appropriate auxiliary operator $\widetilde L$ of order $2\widetilde m>\dmn$ and construct the fundamental solution $\mat E^L$ for $L$ from the fundamental solution~$\mat E^{\widetilde L}$ for~$\widetilde L$. This technique was used in \cite{AusHMT01} in the proof of the Kato conjecture for higher-order operators. Our main results concerning the fundamental solution are summarized as Theorem~\ref{thm:fundamental:high:2} and the following remarks.

This paper may be of some interest to the reader interested only in second-order operators (in the case $\dmn\geq 3$ and in the case of complex coefficients or systems) as our construction extends to the case of operators whose solutions do not satisfy local bounds.


\section{Definitions}
\label{sec:dfn}

Throughout we work with a divergence-form elliptic system of $N$ partial differential equations of order~$2m$ in dimension $\dmn$.

We will often use multiindices in $\N^\dmn$. If $\gamma=(\gamma_1,\dots,\gamma_\dmn)$ is a multiindex, then $\abs{\gamma}=\gamma_1+\gamma_2+\dots+\gamma_\dmn$.
If $\delta=(\delta_1,\dots,\delta_\dmn)$ is another multiindex, then we say that $\delta\leq \gamma$ if $\delta_i\leq \gamma_i$ for all $1\leq i\leq\dmn$, and we say that $\delta<\gamma$ if in addition the strict inequality $\delta_i< \gamma_i$ holds for at least one such~$i$.

We will routinely consider arrays $\mydot F=\begin{pmatrix}F_{j,\gamma}\end{pmatrix}$ indexed by integers~$j$ with $1\leq j\leq N$ and by multiindices~$\gamma$ with $\abs{\gamma}=k$ for some~$k$.
In particular, if $\vec \varphi$ is a vector-valued function with weak derivatives of order up to~$k$, then we view $\nabla^k\vec \varphi$ as such an array, with \begin{equation*}(\nabla^k\vec\varphi)_{j,\gamma}=\partial^\gamma \varphi_j.\end{equation*}
The $L^2$ inner product of two such arrays of numbers $\mydot F$ and $\mydot G$ is given by
\begin{equation*}\bigl\langle \mydot F,\mydot G\bigr\rangle = \sum_{j=1}^N
\sum_{\abs{\gamma}=k}
\overline{F_{j,\gamma}}\, G_{j,\gamma}.\end{equation*}
If $\mydot F$ and $\mydot G$ are two arrays of $L^2$ functions defined in a measurable set~$\Omega\subseteq\R^\dmn$, then the inner product of $\mydot F$ and $\mydot G$ is given by
\begin{equation*}\bigl\langle \mydot F,\mydot G\bigr\rangle_\Omega = \sum_{j=1}^N
\sum_{\abs{\gamma}=k}
\int_{\Omega} \overline{F_{j,\gamma}}\, G_{j,\gamma}
.\end{equation*}

If $E\subset\R^\dmn$ is a set of finite measure, we let $\fint_E f=\frac{1}{|E|}\int_E f$, where $|E|$ denotes Lebesgue measure.
We let $\vec e_k$ be the unit vector in $\R^\dmn$ in the $k$th direction. We let $\mydot e_{j,\gamma}$ be the ``unit array'' corresponding to the multiindex~$\gamma$ and the number~$j$; thus, $\bigl\langle \mydot e_{j,\gamma},\mydot F\bigr\rangle = F_{j,\gamma}$.

We let $L^p(U)$ and $L^\infty(U)$ denote the standard Lebesgue spaces with respect to Lebesgue measure.
We denote the homogeneous Sobolev space $\dot W^p_k(U)$ by
\begin{equation*}\dot W^p_k(U)=\{u:\nabla^k u\in L^p(U)\}\end{equation*}
with the norm $\doublebar{u}_{\dot W^p_k(U)}=\doublebar{\nabla^k u}_{L^p(U)}$. (Elements of $\dot W^p_k(U)$ are then defined only up to adding polynomials of order~$k-1$.) We say that $u\in L^p_{loc}(U)$ or $u\in \dot W^p_{k,loc}(U)$ if $u\in L^p(V)$ or $u\in \dot W^p_k(V)$ for every bounded set $V$ with $\overline V\subset U$. 

\subsection{Elliptic operators}

Let $\mat A = \bigl( A^{jk}_{\alpha\beta} \bigr)$ be an array of measurable coefficients defined on $\R^\dmn$, indexed by integers $1\leq j\leq N$, $1\leq k\leq N$ and by multtiindices $\alpha$, $\beta$ with $\abs{\alpha}=\abs{\beta}=m$. If $\mydot F=\bigl(F_{j,\alpha}\bigr)$ is an array, then $\mat A\mydot F$ is the array given by 
\begin{equation*}(\mat A\mydot F)_{j,\alpha} = \sum_{k=1}^N
\sum_{\abs{\beta}=m}
A^{jk}_{\alpha\beta} F_{k,\beta}.\end{equation*}

Throughout we consider coefficients that satisfy the bound
\begin{align}
\label{eqn:elliptic:bounded:2}
\doublebar{\mat A}_{L^\infty(\R^\dmn)}
&\leq 
	\Lambda
\end{align}
for some $\Lambda>0$.
In our construction of the fundamental solution in Section~\ref{sec:fundamental}, we will consider only operators that satisfy the strict G\r{a}rding inequality
\begin{align}
\label{eqn:elliptic}
\re {\bigl\langle\nabla^m\vec \varphi,\mat A\nabla^m\vec \varphi\bigr\rangle_{\R^\dmn}} 
&\geq 
	\lambda\doublebar{\nabla^m\vec\varphi}_{L^2(\R^\dmn)}^2
\end{align}
for all $\vec \varphi\in \dot W^2_m(\R^\dmn)$ and for some $\lambda>0$ independent of~$\vec\varphi$. In Section~\ref{sec:Caccioppoli} we will consider weaker and stronger versions of the G\r{a}rding inequality.

We let $L$ be the $2m$th-order divergence-form operator associated with~$\mat A$. That is, we say that $L\vec u=\Div_m \mydot F$ in~$\Omega$ in the weak sense if, for every $\vec\varphi$ smooth and compactly supported in~$\Omega$, we have that
\begin{equation}
\label{eqn:L}
\bigl\langle\nabla^m\vec\varphi, \mat A\nabla^m\vec u\bigr\rangle_\Omega
=
\bigl\langle\nabla^m\vec\varphi, \mydot F\bigr\rangle_\Omega,
\end{equation}
that is, we have that
\begin{equation*}
\sum_{j=1}^N \sum_{k=1}^N
\sum_{\abs{\alpha}=\abs{\beta}=m}
\int_{\Omega}\partial^\alpha \bar \varphi_j\, A^{jk}_{\alpha\beta}\,\partial^\beta u_k
=
\sum_{j=1}^N
\sum_{\abs{\alpha}=m}
\int_{\Omega} \partial^\alpha \bar \varphi_j \, F_{j,\alpha}
.
\end{equation*}
In particular, if the left-hand side is zero for all such~$\vec\varphi$ then we say that $L\vec u=0$.

If $\mat A$ is such an array of coefficients, we let the adjoint array $\mat A^*$ be given by $(A^*)^{jk}_{\alpha\beta}=\overline{A^{kj}_{\beta\alpha}}$; we then let $L^*$ be the operator associated with~$\mat A^*$.

Throughout the paper we will let $C$ denote a constant whose value may change from line to line, but which depends only on the dimension $\dmn$, the ellipticity constants $\lambda$ and $\Lambda$ in the bounds~\eqref{eqn:elliptic:bounded:2} and \eqref{eqn:elliptic} (or variants thereof), and the order $2m$ of the operator~$L$. Any other dependencies will be indicated explicitly.

\section{The Caccioppoli inequality}
\label{sec:Caccioppoli}

In this section we will generalize the Caccioppoli inequality \eqref{eqn:Caccioppoli:2} to the case of higher-order elliptic systems.

We will begin with the following lemma.
\begin{lem}
\label{lem:Caccioppoli:1}
Let $L$ be the operator of order $2m$ associated to the coefficients~$\mat A$, where $\mat A$ satisfies the bound~\eqref{eqn:elliptic:bounded:2} and the weak G\r{a}rding inequality
\begin{align}
\label{eqn:elliptic:AusQ00}
\re {\bigl\langle\nabla^m\vec \varphi,\mat A\nabla^m\vec \varphi\bigr\rangle_{\R^\dmn}} 
&\geq 
	\lambda\doublebar{\nabla^m\vec\varphi}_{L^2(\R^\dmn)}^2
	-\delta\doublebar{\vec\varphi}_{L^2(\R^\dmn)}^2
\end{align}
for some $\lambda>0$ and some $\delta>0$, and for all smooth, compactly supported functions~$\vec\varphi$.

Let $x_0\in\R^\dmn$ and let $R>0$.
Suppose that $\vec u\in \dot W^2_m(B(x_0,R))$, that $\mydot F\in L^2(B(x_0,R))$, and that one of the following two conditions holds.
\begin{description}
\eqitem\label{eqn:Caccioppoli:interior} $L\vec u=\Div_m\mydot F$ in $\Omega=B(x_0,R)$, or
\eqitem\label{eqn:Caccioppoli:Dirichlet} $L\vec u=\Div_m\mydot F$ in some domain $\Omega\subsetneq B(x_0,R)$, and $\vec u$ lies in the closure in $\dot W^2_m(B(x_0,R))$ of $\{\vec\varphi\in C^\infty(\R^\dmn):\vec\varphi\equiv 0 $ in $B(x_0,R)\setminus\Omega\}$.
\end{description}

Then, for any $0<r<R$, we have that
\begin{multline}
\label{eqn:Caccioppoli:1}
\int_{\Omega\cap B(x_0,r)} \abs{\nabla^m \vec u}^2 
\\\leq \sum_{i=0}^{m-1} \frac{C}{(R-r)^{2m-2i}}
	\int_{\Omega\setminus B(x_0,r)} \abs{\nabla^i \vec u}^2
+C \int_\Omega \abs{\mydot F}^2
+C\delta \int_\Omega \abs{\vec u}^2
\end{multline}
where $C$ is a constant depending only on the dimension~$\dmn$, the order~$2m$ of the elliptic operator~$L$ and the numbers $\lambda$ and $\Lambda$ in the bounds \eqref{eqn:elliptic:bounded:2} and~\eqref{eqn:elliptic:AusQ00}.
\end{lem}
In Theorem~\ref{thm:Caccioppoli:2} we will strengthen this lemma by replacing the sum on the right-hand side by the $i=0$ term alone.
Our Theorem~\ref{thm:Caccioppoli:2} will thus be stronger than the bound \eqref{eqn:Caccioppoli:AusQ00} of \cite{AusQ00}; we have chosen to follow the example of \cite{AusQ00} and establish the Caccioppoli inequality for operators that satisfy the weak G\r{a}rding inequality \eqref{eqn:elliptic:AusQ00}, as well as operators that satisfy the strong G\r{a}rding inequality~\eqref{eqn:elliptic}.

Lemma~\ref{lem:Caccioppoli:1} was proven in \cite{Cam80} in the interior case \eqref{eqn:Caccioppoli:interior} for coefficients~$\mat A$ that satisfy the strong pointwise G\r{a}rding inequality
\begin{equation}\label{eqn:elliptic:pointwise}
\re \bigl\langle \mydot\eta, \mat A(x)\mydot \eta \bigr\rangle
\geq 
\lambda
\bigl\langle \mydot\eta, \mydot \eta \bigr\rangle
\quad\text{for almost every $x\in\R^\dmn$ and any array $\mydot\eta$.}
\end{equation}
Thus the main new result of Lemma~\ref{lem:Caccioppoli:1} is the case  \eqref{eqn:Caccioppoli:Dirichlet}, which corresponds to zero Dirichlet boundary values.

In the higher-order case, the condition that $\vec u$ have zero Neumann boundary values along $\partial\Omega\cap B(x_0,2r)$ may best be expressed by the following condition.
\begin{description}
\eqitem\label{eqn:Caccioppoli:Neumann}
$\vec u\in \dot W^2_m(B(x_0,R))$, and the equation 
\begin{equation*}\bigl\langle \nabla^m\vec\varphi,\mydot F\bigr\rangle_\Omega=\bigl\langle \nabla^m\vec\varphi,\mat A
\nabla^m\vec u\bigr\rangle_\Omega.\end{equation*}
is true for all 
$\vec\varphi$ is smooth and supported in $B(x_0,R)$, not only all $\vec \varphi$ supported in~$\Omega$. 
\end{description}
We will discuss the meaning of the Neumann boundary values of a solution extensively in a forthcoming paper.

\begin{lem}
\label{lem:Caccioppoli:Neumann}
If $L\vec u=\Div_m\mydot F$ in~$\Omega\subset B(x_0,R)$ and $\vec u$ satisfies the Neumann boundary condition~\eqref{eqn:Caccioppoli:Neumann}, then the conclusion \eqref{eqn:Caccioppoli:1} of Lemma~\ref{lem:Caccioppoli:1} is still true provided that the coefficients~$\mat A$ associated with the operator~$L$ satisfy the bound \eqref{eqn:elliptic:bounded:2} and the local G\r{a}rding inequality
\begin{align}
\label{eqn:elliptic:local}
\re \bigl\langle\nabla^m\vec \varphi,\mat A\nabla^m\vec \varphi\bigr\rangle_{\Omega} 
&\geq 
	\lambda\doublebar{\nabla^m\vec\varphi}_{L^2(\Omega)}^2
	-\delta\doublebar{\vec\varphi}_{L^2(\Omega)}^2
\end{align}
for all $\vec\varphi\in \dot W^2_m(\R^\dmn)$.
\end{lem}
Notice that the pointwise ellipticity condition \eqref{eqn:elliptic:pointwise} implies the local G\r{a}rding inequality~\eqref{eqn:elliptic:local}.

In all cases we assume that $\vec u$ is defined in the ball $B(x_0,R)$; equivalently, we assume that we may extend $\vec u$ from $\Omega$ to the ball. 
This extension is very natural in the interior or Dirichlet cases but must be explicitly assumed in the Neumann case.
If $\Omega$ is a Lipschitz domain and $\nabla^m \vec u\in L^2(\Omega)$, then by a well-known result of Calder\'on and Stein, an extension of $\vec u$ to $B(x_0,R)$ (indeed, to $\R^\dmn$) exists. Such extensions are also guaranteed to exist under weaker conditions on~$\Omega$; see, for example, \cite{Jon81}. 

Notice further that in the interior and Neumann cases \eqref{eqn:Caccioppoli:interior} and~\eqref{eqn:Caccioppoli:Neumann} the conclusion \eqref{eqn:Caccioppoli:1} remains valid if we modify $\vec u$ by adding a polynomial of order~$m-1$; however, this is not true in the Dirichlet case \eqref{eqn:Caccioppoli:Dirichlet}, as in this case we must maintain the condition $\vec u\equiv 0$ in $B(x_0,R)\setminus\Omega)$.

\begin{proof}[Proof of Lemmas~\ref{lem:Caccioppoli:1} and~\ref{lem:Caccioppoli:Neumann}]
Let $\varphi$ be a smooth, nonnegative real-valued test function supported in $B(x_0,R)$ and identically equal to 1 in $B(x_0,r)$.
We require $\abs{\nabla^k\varphi}\leq C_k (R-r)^{-k}$.

Observe that $\vec\psi=\varphi^{4m} \vec u$ is a function supported in $B(x_0,R)$ with $\nabla^m\vec\psi\in L^2(B(x_0,R))$. By definition of $L\vec u$ or condition~\eqref{eqn:Caccioppoli:Neumann}, and by density of smooth functions, we have that
\begin{align*}
\bigl\langle \nabla^m(\varphi^{4m} \vec u),\mydot F\bigr\rangle_\Omega
&= 
	\bigl\langle \nabla^m (\varphi^{4m}\vec u), \mat A\nabla^m\vec u \bigr\rangle_\Omega
.\end{align*}

By the product rule, there are constants $a_{\alpha,\gamma}$ such that \begin{equation*}\partial^\alpha (w\,v) =\sum_{\gamma\leq \alpha} a_{\alpha,\gamma} \partial^\gamma w\,\partial^{\alpha-\gamma} v\end{equation*}
for all suitably differentiable functions $v$ and~$w$.
Notice that $a_{\alpha,0}=a_{\alpha,\alpha}=1$.

By definition of the inner product, we have that 
\begin{align*}
\abs[big]{\bigl\langle \nabla^m(\varphi^{4m} \vec u),\mydot F\bigr\rangle_\Omega}
&=
\abs[bigg]{\sum_{j=1}^N
\sum_{\abs{\alpha}=m}\int_\Omega\partial^\alpha(\varphi^{4m} \bar u_j)\mydot F_{j,\alpha}
}
.\end{align*}
Applying the product rule to $\varphi^{4m}\vec u=(\varphi^{2m})( \varphi^{2m}\vec u)$, we see that
\begin{align*}
\abs[big]{\bigl\langle \nabla^m(\varphi^{4m} \vec u),\mydot F\bigr\rangle_\Omega}
&\leq
	\abs[bigg]{\sum_{j=1}^N
	\sum_{\abs{\alpha}=m}\int_\Omega\partial^\alpha(\varphi^{2m} \bar u_j)\,\varphi^{2m} F_{j,\alpha}
	}
	\\&\qquad
	+\abs[bigg]{\sum_{j=1}^N
	\sum_{\abs{\alpha}=m}\sum_{\gamma<\alpha} a_{\alpha,\gamma} \int_\Omega\partial^{\alpha-\gamma} (\varphi^{2m})\,\partial^\gamma(\varphi^{2m} \bar u_j)\mydot F_{j,\alpha}
	}
.\end{align*}
Thus
\begin{align*}
\abs[big]{\bigl\langle \nabla^m(\varphi^{4m} \vec u),\mydot F\bigr\rangle_\Omega}
&\leq
	\doublebar{\nabla^m(\varphi^{2m}\vec u)}_{L^2(\Omega)} \doublebar{\mydot F}_{L^2(\Omega)}
	\\&\qquad
	+ C \sum_{i=0}^{m-1} \frac{1}{(R-r)^{m-i}}
	\doublebar{\nabla^i u}_{L^2({\Omega\setminus B(x_0,r)})}  \doublebar{\mydot F}_{L^2(\Omega)}
.\end{align*}

We now consider the right-hand side. We have that
\begin{align*}
\bigl\langle \nabla^m (\varphi^{4m}\vec u), \mat A\nabla^m\vec u \bigr\rangle_\Omega
&= 
		\sum_{j,k,\alpha,\beta}
		\int_\Omega 		
		\sum_{\gamma<\alpha}
		a_{\alpha,\gamma}
		\partial^{\alpha-\gamma}(\varphi^{2m})
		\partial^\gamma (\varphi^{2m} \bar u_j) 
		\,A_{\alpha\beta}^{j,k} \,\partial^\beta u_k
	\\&\qquad+
		\sum_{j,k,\alpha,\beta}
		\int_\Omega 		
		\varphi^{2m}
		\partial^\alpha (\varphi^{2m} \bar u_j) 
		\,A_{\alpha\beta}^{j,k} \,\partial^\beta u_k
\end{align*}
where the sums are taken over all $j$, $k$, $\alpha$, $\beta$ with $1\leq j\leq N$, $1\leq k\leq N$ and $\abs{\alpha}=\abs{\beta}=m$.
Now, we may write
\begin{equation*}\sum_{\gamma<\alpha}
		a_{\alpha,\gamma}
		\partial^{\alpha-\gamma}(\varphi^{2m})
		\partial^\gamma (\varphi^{2m} \bar u_j) 
=	
	\sum_{\zeta<\alpha}
	\varphi^{2m}
	\Phi_{\alpha,\zeta}
	\partial^\zeta \bar u_j
\end{equation*}
for some functions $\Phi_{\alpha,\zeta}$ supported in $B(x_0,R)\setminus B(x_0,r)$ with
$\abs{\Phi_{\alpha,\zeta}}\leq C(R-r)^{\abs{\zeta}-\abs{\alpha}}$. 
Therefore
\begin{align*}
\bigl\langle \nabla^m (\varphi^{4m}\vec u), \mat A\nabla^m\vec u \bigr\rangle_\Omega
&= 
		\sum_{j,k,\alpha,\beta}
		\int_\Omega 				
		\partial^\alpha (\varphi^{2m} \bar u_j) 
		\,A_{\alpha\beta}^{j,k} \,(\varphi^{2m}\partial^\beta u_k)
	\\&\quad+
		\sum_{j,k,\alpha,\beta}
		\int_{\Omega\setminus B(x_0,r)} 		
		\sum_{\zeta<\alpha}		
		\Phi_{\alpha,\zeta}
		\partial^\zeta \bar u_j
		\,A_{\alpha\beta}^{j,k} \,(\varphi^{2m}\partial^\beta u_k)
.\end{align*}
We rewrite the two terms $\varphi^{2m}\partial^\beta u_k$ to see that
\begin{multline*}
\bigl\langle \nabla^m (\varphi^{4m}\vec u), \mat A\nabla^m\vec u \bigr\rangle_\Omega
\\\begin{aligned}
&= 
		\bigl\langle \nabla^m (\varphi^{2m}\vec u), \mat A\nabla^m (\varphi^{2m}\vec u)\bigr\rangle_\Omega
	\\&\qquad+
		\sum_{j,k,\alpha,\beta}
		\int_{\Omega}	
		\sum_{\zeta<\alpha}
		\Phi_{\alpha,\zeta}
		\partial^\zeta \bar u_j
		\,A_{\alpha\beta}^{j,k} \,\partial^\beta(\varphi^{2m} u_k)
	\\&\qquad-
		\sum_{j,k,\alpha,\beta}
		\int_\Omega 		
		\sum_{\gamma<\beta}
		a_{\beta,\gamma}
		\Bigl(\partial^\alpha (\varphi^{2m} \bar u_j) 
		+\sum_{\zeta<\alpha}
		\Phi_{\alpha,\zeta}
		\partial^\zeta \bar u_j
		\Bigr)
		\,A_{\alpha\beta}^{j,k} \,\partial^{\beta-\gamma}(\varphi^{2m})\,\partial^\gamma u_k
.\end{aligned}\end{multline*}
Observe that the integrands in the second and third terms are zero in $B(x_0,r)$.

By the G\r{a}rding inequality \eqref{eqn:elliptic:AusQ00} or  \eqref{eqn:elliptic:local},
\begin{equation*}
\lambda \int_\Omega 
		\abs{\nabla^m(\varphi^{2m} \vec u) }^2
	\leq
		\re	\bigl\langle\nabla^m(\varphi^{2m}\vec u),\mat A\nabla^m(\varphi^{2m}\vec u)\bigr\rangle_\Omega
	+ \delta\doublebar{\varphi^{2m}\vec u}_{L^2(\Omega)}^2
.\end{equation*}
Thus
\begin{align*}
\lambda \int_\Omega 
		\abs{\nabla^m(\varphi^{2m}\vec u) }^2
&\leq	
	\abs{
	\bigl\langle \nabla^m (\varphi^{4m}\vec u), \mat A\nabla^m\vec u \bigr\rangle_\Omega
	}+ \delta\doublebar{\vec u}_{L^2(\Omega)}^2
	\\&\qquad+
	\doublebar{\nabla^m(\varphi^{2m}\vec u)}_{L^2(\Omega)} 
	\sum_{i=0}^{m-1} \frac{C}{(R-r)^{m-i}}
	\doublebar{\nabla^i u}_{L^2({\Omega\setminus B(x_0,r)})} 
	\\&\qquad+ 	
	\sum_{i=0}^{m-1} \frac{C}{(R-r)^{m-i}}
	\doublebar{\nabla^i u}_{L^2({\Omega\setminus B(x_0,r)})}^2
.\end{align*}
Recalling that 
\begin{align*}
\abs{
\bigl\langle \nabla^m (\varphi^{4m}\vec u), \mat A\nabla^m\vec u \bigr\rangle_\Omega
}
&=\abs{
\bigl\langle \nabla^m (\varphi^{4m}\vec u), \mydot F\bigr\rangle_\Omega
}
\\&\leq 
	\doublebar{\nabla^m(\varphi^{2m}\vec u)}_{L^2(\Omega)} \doublebar{\mydot F}_{L^2(\Omega)}
	\\&\qquad
	+ C \sum_{i=0}^{m-1} \frac{1}{(R-r)^{m-i}}
	\doublebar{\nabla^i u}_{L^2({\Omega\setminus B(x_0,r)})}  \doublebar{\mydot F}_{L^2(\Omega)}
\end{align*}
we may derive the desired bound on $\doublebar{\nabla^m\vec u}_{L^2(\Omega\cap B(x_0,r))}$.
\end{proof}

We now wish to improve this inequality to a bound in terms of $\doublebar{u}_{L^2}$ rather than in terms of all of the lower-order derivatives. This will be done by the following theorem and its corollaries.

\begin{thm} \label{thm:Caccioppoli:2} 
Let $x_0\in\R^\dmn$ and let $R>0$. Let $\vec u\in \dot W^2_m(B(x_0,R))$ be a function that satisfies the inequality
\begin{equation}
\label{eqn:Caccioppoli:1:1}
\int_{B(x_0,\rho)} \abs{\nabla^{m} \vec u}^2 \leq 
\sum_{i=0}^{m-1}\frac{C_0}{(r-\rho)^{2m-2i} } \int_{B(x_0,r)\setminus B(x_0,\rho)} \abs{\nabla^{i} \vec u}^2 
+ F
\end{equation}
whenever $0<\rho<r<R$, for some number $F>0$.

Then $u$ satisfies the stronger inequality
\begin{equation}
\label{eqn:Caccioppoli:m}
\int_{B(x_0,r)} \abs{\nabla^m \vec u}^2 \leq 
\frac{C}{(R-r)^{2m}} \int_{B(x_0,R)\setminus B(x_0,r)} \abs{\vec u}^2 
+ C F
\end{equation}
for some constant $C$ depending only on $m$, the dimension $\dmn$ and the constant~$C_0$.

Furthermore, if $0\leq j\leq m$, then $u$ satisfies
\begin{equation}
\label{eqn:Caccioppoli:lower}
\int_{B(x_0,r)} \abs{\nabla^j \vec u}^2 \leq 
\frac{C}{(R-r)^{2j}} \int_{B(x_0,R)} \abs{\vec u}^2 
+ C R^{2m-2j}F
.\end{equation}
\end{thm}

Notice that in the bound~\eqref{eqn:Caccioppoli:m}, the right-hand side involves the quantity $\abs{u}^2$ integrated over an annulus ${B(x_0,R)\setminus B(x_0,r)}$, while in the bound \eqref{eqn:Caccioppoli:lower} $\abs{u}^2$ is integrated over the full ball ${B(x_0,R)}$. It is possible to use the Poincar\'e inequality and the bound \eqref{eqn:Caccioppoli:m} to improve the bound \eqref{eqn:Caccioppoli:lower} to an estimate involving the integral of $\abs{u}^2$ over an annulus, but this comes at a cost of introducing powers of $(R-r)/r$, and so we have chosen to state the bound \eqref{eqn:Caccioppoli:lower} as above.

Combined with Lemma~\ref{lem:Caccioppoli:1}, we immediately have the following corollaries.
\begin{cor}\label{cor:Caccioppoli:interior}
Let $x_0\in\R^\dmn$ and let $R>0$.
Suppose that $L\vec u=\Div_m\mydot F$ in $B(x_0,R)$, for some operator~$L$ of order~$2m$ that satisfies the bounds \eqref{eqn:elliptic:bounded:2} and~\eqref{eqn:elliptic:AusQ00}, some $\vec u\in\dot W^2_m(B(x_0,R))$, and some $\mydot F\in L^2(B(x_0,R))$.
If $0<r<R$ and $0\leq j\leq m$, then
\begin{align*}
\int_{B(x_0,r)} \abs{\nabla^j \vec u}^2 
&\leq 
\frac{C}{(R-r)^{2j} } \int_{B(x_0,R)} \abs{\vec u}^2
+ C R^{2m-2j}\int_{B(x_0,R)} \bigl(\abs{\mydot F}^2+\delta\abs{\vec u}^2\bigr)
,\\
\int_{B(x_0,r)} \abs{\nabla^m \vec u}^2 
&\leq 
\frac{C}{(R-r)^{2m} } \int_{B(x_0,R)\setminus B(x_0,r)} \abs{\vec u}^2
+ C \int_{B(x_0,R)} \bigl(\abs{\mydot F}^2+\delta\abs{\vec u}^2\bigr)
.\end{align*}
\end{cor}

Recall that if we allow a term of the form $\varepsilon \doublebar{\nabla^m\vec u}_{L^2(B(x_0,R))}^2$ on the right-hand side, then this corollary was proven in \cite{AusQ00} in the homogeneous case $L\vec u=0$.

\begin{cor}\label{cor:Caccioppoli:Dirichlet}
Let $x_0\in\R^\dmn$ and let $R>0$, and let $\Omega\subset B(x_0,R)$.
Suppose that $L\vec u=\Div_m\mydot F$ in $\Omega$, for some operator~$L$ of order~$2m$ that satisfies the bounds \eqref{eqn:elliptic:bounded:2} and~\eqref{eqn:elliptic:AusQ00}, some $\vec u\in\dot W^2_m(\Omega)$, and some $\mydot F\in L^2(\Omega)$. Suppose in addition that $\vec u$ may be extended by zero to all of $B(x_0,R)$, in the sense of condition~\eqref{eqn:Caccioppoli:Dirichlet} of Lemma~\ref{lem:Caccioppoli:1}.

If $0<r<R$ and $0\leq j\leq m$, then
\begin{align*}
\int_{B(x_0,r)\cap\Omega} \abs{\nabla^j \vec u}^2 
\leq 
\frac{C}{(R-r)^{2j} } \int_{\Omega\cap B(x_0,R)} \abs{\vec u}^2
+ C R^{2m-2j}\int_\Omega \bigl(\abs{\mydot F}^2+\delta\abs{\vec u}^2\bigr)
,\\
\int_{B(x_0,r)\cap\Omega} \abs{\nabla^m \vec u}^2 
\leq 
\frac{C}{(R-r)^{2m} } \int_{\Omega\cap B(x_0,R)\setminus B(x_0,r)} \abs{\vec u}^2
+ C \int_\Omega \bigl(\abs{\mydot F}^2+\delta\abs{\vec u}^2\bigr)
.\end{align*}
\end{cor}

Our methods will not allow us to improve upon Lemma~\ref{lem:Caccioppoli:Neumann} in the case of Neumann boundary data.

\begin{proof}[Proof of Theorem~\ref{thm:Caccioppoli:2}]
Let $A(r,\zeta)$ denote either the annulus $B(x_0,r+\zeta)\setminus B(x_0,r-\zeta)$, or simply the ball $B(x_0,r+\zeta)$, depending on whether we are establishing the bound \eqref{eqn:Caccioppoli:m} on $\nabla^m \vec u$ or the bound \eqref{eqn:Caccioppoli:lower} on~$\nabla^k \vec u$.

Consider the following claim.

\textbf{Claim.} 
If $1\leq k\leq m$, and if $R/2<r<R$ and $0<\zeta<\min(R-r,r)$, then 
\begin{equation*}
\int_{A(r,\zeta)} \abs{\nabla^{k} \vec u}^2 \leq 
\sum_{i=0}^{k-1}\frac{C_k}{(\xi-\zeta)^{2k-2i} } \int_{A(r,\xi)} \abs{\nabla^{i}\vec u}^2 
+  R^{2m-2k}F
.\end{equation*}

If this claim is true for all such~$k$, then clearly the bound \eqref{eqn:Caccioppoli:lower} is valid. To establish the bound \eqref{eqn:Caccioppoli:m}, we combine the above claim with the assumed bound~\eqref{eqn:Caccioppoli:1:1}; it is this that allows us to bound $\nabla^m\vec u$ by the integral of $\abs{\vec u}^2$ over an annulus rather than a ball.

Thus we need only prove the claim. That the claim is true for $k=m$ follows by our assumption~\eqref{eqn:Caccioppoli:1:1}. We work by induction. Suppose that the claim is true for some $k+1<m$; we will show that it is valid for $k$ as well.

Let $A_j=A(r,\rho_j)$, where $\zeta=\rho_0<\rho_1<\dots<\xi$ for some sequence $\{\rho_j\}_{j=0}^\infty$ to be chosen momentarily. Let $\delta_j=\rho_{j+1}-\rho_j$, and let $\widetilde A_j=A(r,\rho_j+\delta_j/2)$, so $A_j\subset\widetilde A_j\subset A_{j+1}$. Let $\varphi_j$ be smooth, supported in $\widetilde A_j$, and identically equal to 1 in $A_j$; we may require that $\doublebar{\nabla\varphi_k}\leq C/\delta_j$ and $\doublebar{\nabla^2\varphi_k}\leq C/\delta_j^2$ for some absolute constant~$C$.

Now, for any $j\geq 0$,
\begin{align*}
\int_{A_j} \abs{\nabla^{k} \vec u}^2
&\leq\int_{\widetilde A_j} \abs{\nabla(\varphi_j\nabla^{k-1} \vec u)}^2
.\end{align*}
By Plancherel's theorem, if $f\in \dot W^2_2(\R^\dmn)\cap L^2(\R^\dmn)$ then \begin{equation*}\doublebar{\nabla f}_{L^2(\R^\dmn)}^2\leq C\doublebar{\nabla^2 f}_{L^2(\R^\dmn)} \doublebar{f}_{L^2(\R^\dmn)}.\end{equation*}
We will apply this inequality to $f=(\varphi_j\nabla^{k-1} \vec u)$; it is this step that fails in the case of Neumann boundary data. We have that
\begin{align*}
\int_{A_j} \abs{\nabla^{k} \vec u}^2
&\leq
	C\biggl(\int_{\widetilde A_j} \abs{\nabla^2(\varphi_j\nabla^{k-1} \vec u)}^2\biggr)^{1/2}
	\biggl(\int_{\widetilde A_j} \abs{\varphi_j\nabla^{k-1} \vec u}^2\biggr)^{1/2}
\\&\leq
	C\biggl(\int_{\widetilde A_j} 
	\abs{\nabla^{k+1}\vec  u}^2 
	+\frac{1}{\delta_j^2}\abs{\nabla^{k}\vec  u}^2
	+\frac{1}{\delta_j^4}\abs{\nabla^{k-1}\vec  u}^2\biggr)^{1/2}
	\biggl(\int_{\widetilde A_j} \abs{\nabla^{k-1}\vec  u}^2\biggr)^{1/2}\!
.\end{align*}
Applying the claim to bound $\abs{\nabla^{k+1}\vec u}^2$, we see that
\begin{align*}
\int_{A_j} \abs{\nabla^{k}\vec  u}^2
&\leq
	\biggl(
	\sum_{i=0}^{k} \frac{ C_k}{\delta_j^{2k+2-2i} } \int_{A_{j+1}}\abs{\nabla^{i}\vec u}^2 
	+CR^{2m-2k-2}F
	\biggr)^{1/2}
	\biggl(\int_{\widetilde A_j} \abs{\nabla^{k-1} \vec u}^2\biggr)^{1/2}
.\end{align*}
We move a factor of $C_k/\delta_j^2$ from the first term to the second, and then use the inequality $\sqrt a\sqrt b \leq (1/2)a + (1/2)b$ to see that
\begin{align*}
\int_{A_j} \abs{\nabla^{k}\vec  u}^2
&\leq
	\frac{1}{2}\sum_{i=0}^{k}\frac{1}{\delta_j^{2k-2i} } \int_{A_{j+1}} \abs{\nabla^{i}\vec u}^2 
	+\frac{1}{2}R^{2m-2k}F
	+
	\frac{C_k}{\delta_j^2}\int_{\widetilde A_j} \abs{\nabla^{k-1} u}^2
.\end{align*}
Separating out the term $i=k$, we see that
\begin{align*}
\int_{A_j} \abs{\nabla^{k} \vec u}^2
&\leq
	C_k\sum_{i=0}^{k-1}\frac{1}{\delta_j^{2k-2i} } \int_{A_{j+1}} \abs{\nabla^{i}\vec u}^2 
	+\frac{1}{2}R^{2m-2k}F
	+\frac{1}{2}\int_{A_{j+1}} \abs{\nabla^k\vec u}^2 	
.\end{align*}
This bound is valid for all $j>0$. We may iterate to see that
\begin{align*}
\int_{A_0} \abs{\nabla^{k} \vec u}^2
&\leq
	\sum_{j=1}^\infty 2^{-(j-1)}
	\biggl(
	C_k\sum_{i=0}^{k-1}\frac{1}{\delta_j^{2k-2i} } \int_{A_{j}} \abs{\nabla^{i}\vec u}^2 
	+\frac{1}{2}R^{2m-2k}F\biggr)
\\&\leq
	C_1\sum_{i=0}^{k-1}
	\biggl(\sum_{j=1}^\infty 2^{-(j-1)}\frac{1}{\delta_j^{2k-2i} } \biggr) \int_{A_\infty} \abs{\nabla^{i}\vec u}^2 
	+R^{2m-2k}F
.\end{align*}
Now, choose $\rho_j=\zeta+(\xi-\zeta)(1-\tau)\sum_{i=1}^j\tau^i$ for some $0<\tau<1$. Then $\rho_0=\zeta$ and $\lim_{j\to \infty} \rho_j=\xi$. So
\begin{align*}
\int_{A_0} \abs{\nabla^{k} \vec u}^2
&\leq
	C_{k,\tau}\sum_{i=0}^{k-1}
	\biggl(\sum_{j=1}^\infty 
	\frac{1}{(2\tau^{2k-2i})^j}
	\frac{1}{(\xi-\zeta)^{2k-2i} } \biggr) \int_{A_\infty} \abs{\nabla^{i}\vec u}^2 
	+R^{2m-2k}F
.\end{align*}
Choosing $\tau$ so that $2\tau^{2k}>1$ and $\tau<1$, we see that the sum in $j$ converges and the proof is complete.
\end{proof}

\section{Meyers's reverse H\"older inequality for gradients}
\label{sec:Meyers}

In this section we will generalize Meyers's reverse H\"older inequality \eqref{eqn:Meyers:2} to the higher-order case. We will use many of the techniques of the second-order case. The interior and Dirichlet boundary versions of this inequality are stated in the following theorem; the Neumann boundary version is stated below in Theorem~\ref{thm:Meyers:Neumann}.

\begin{thm}
\label{thm:Meyers} 
Let $L$ be an operator of order~$2m$ that satisfies the bounds \eqref{eqn:elliptic:bounded:2} and~\eqref{eqn:elliptic:AusQ00}. Let $c_\Omega>0$.
Then there is some number $p^+=p^+_L>2$ depending only on the standard constants and the number $c_\Omega$ such that the following statement is true.

Let $x_0\in\R^\dmn$ and let $R>0$.
Suppose that $\vec u\in \dot W^2_m(B(x_0,R))$, that $\mydot F\in L^2(B(x_0,R))$, and that either
\begin{description}
\eqitem $L\vec u=\Div_m\mydot F$ in $\Omega=B(x_0,R)$, or
\eqitem $L\vec u=\Div_m\mydot F$ in some domain $\Omega\subsetneq B(x_0,R)$, and $\vec u$ lies in the closure in $\dot W^2_m(B(x_0,R))$ of $\{\vec\varphi\in C^\infty(\R^\dmn):\vec\varphi\equiv 0 $ in $B(x_0,R)\setminus\Omega\}$.
Furthermore, if $x\in\partial\Omega$ and $\rho>0$, then $\abs{B(x_0,\rho)\setminus\Omega}\geq c_\Omega\rho^\dmn$, where $\abs{E}$ denotes the Lebesgue measure of~$E$.
\end{description}

Suppose that $0<p\leq 2<q<p^+$. Then
\begin{align}
\label{eqn:Meyers}
\biggl(\int_{B(x_0,r)\cap\Omega}\abs{\nabla^m \vec u}^q\biggr)^{1/q}
&\leq 
	\frac{C(c_\Omega,p,q)}{(R-r)^{\pdmn/p-\pdmn/q}}
	\biggl(\int_{\Omega}\abs{\nabla^m \vec u}^p\biggr)^{1/p} 
	\\&\qquad\nonumber
	+ C(c_\Omega,p,q) \biggl(\int_{\Omega}\abs{\mydot F}^q+\delta^{q/2}\abs{\vec u}^q\biggr)^{1/q}
\end{align}
for some constant $C(c_\Omega,p,q)$ depending only on $p$, $q$, $c_\Omega$ and the standard parameters.

We may also bound the lower-order derivatives. Suppose that $m-\pdmn/2 < m-k < m$ and that $0\leq m-k$.
Let $0<p\leq 2\leq q<\min(p^+_L,\pdmn/k)$. Then 
\begin{align}\label{eqn:Meyers:lower}
\biggl(\int_{B(x_0,r)\cap\Omega}\abs{\nabla^{m-k} \vec u}^{q_k}\biggr)^{1/{q_k}}
&\leq 
	\frac{C(c_\Omega,p,q)}{(R-r)^{\pdmn/p_k-\pdmn/{q_k}}}
	\biggl(\int_{\Omega}\abs{\nabla^{m-k} \vec u}^{p_k}\biggr)^{1/p_k} 
	\\&\qquad\nonumber
	+ C(c_\Omega,p,q) R^{k}\biggl(\int_{\Omega}\abs{\mydot F}^{q}+\delta^{q/2}\abs{\vec u}^q\biggr)^{1/q}
\end{align}
where $q_k=q\,\pdmn/(\dmn-k\,q)$ and  $p_k=p\,\pdmn/(\dmn-k\,p)$. (Notice that the condition $0<p\leq 2\leq q<\min(p^+_L,\pdmn/k)$ is equivalent to the condition $0<p_k\leq 2_k\leq q_k<p^+_k$, where $2_k=2\,\pdmn/(\dmn-2k)$ and  $p^+_k=p^+_L\,\pdmn/(\dmn-k\,p^+_L)$ if $\dmn>k\,p^+_L$ and $p^+_k=\infty$ if $\dmn\leq k\,p^+_L$.)

Finally, if $0\leq m-k \leq m-\pdmn/2$ and $0<p<\infty$, then $\nabla^{m-k}\vec u$ is H\"older continuous and satisfies the bound
\begin{align}\label{eqn:Meyers:lowest}
\sup_{B(x_0,r)\cap\Omega} \abs{\nabla^{m-k} \vec u} 
&\leq 
	\frac{C(p,q)}{(R-r)^{\pdmn/p}}
	\biggl(\int_{\Omega}\abs{\nabla^{m-k} \vec u}^p\biggr)^{1/p} 
	\\&\qquad\nonumber
	+ C(p,q) R^{k-\pdmn/q}\biggl(\int_{\Omega}\abs{\mydot F}^{q}+\delta^{q/2}\abs{\vec u}^q\biggr)^{1/q}
\end{align}
provided that $0<p\leq \infty$ and that either $q\geq 2$ and $k>\pdmn/2$ or $q>2$ and $k\geq\pdmn/2$.
\end{thm}
Of course if $p>q$, then we may use H\"older's inequality to bound $\doublebar{\nabla^{m-k}\vec u}_{L^q}$ by $\doublebar{\nabla^{m-k}\vec u}_{L^p}$; however, we then no longer have the coefficient $(R-r)^{\pdmn/q-\pdmn/p}$.
In the interior case $\Omega=B(x_0,R)$, the bound \eqref{eqn:Meyers} with $p=2$ was proven in \cite{AusQ00} in the homogeneous case $L\vec u=0$, and in \cite{Cam80} under the strong pointwise G\r{a}rding inequality; the lower-order bounds \eqref{eqn:Meyers:lower} and \eqref{eqn:Meyers:lowest} are relatively straightforward consequences of the bound \eqref{eqn:Meyers} but it will be convenient later to have them stated explicitly.

We will prove Theorem~\ref{thm:Meyers} as in the second-order case; we will need the following lemmas. The first two given lemmas are standard in the theory of Sobolev spaces; see, for example, \cite[Section~5.6.3]{Eva98}.

\begin{lem}\label{lem:GNS:ball}\textup{(The Gagliardo-Nirenberg-Sobolev inequality in balls).}
Let $x_0\in\R^\dmn$ and let $\rho>0$.
Suppose that $1\leq q<\dmn$, that $1\leq k<\pdmn/q$, and that $\nabla^k v\in L^q(B(x_0,\rho))$. Let $q_k=q\,\pdmn/(\dmn-k\,q)$. 

Then $v\in L^{q_k}(B(x_0,\rho))$. More precisely,
\begin{equation*}\doublebar{v}_{L^{q_k}(B(x_0,\rho))} \leq C(q,k) \sum_{i=0}^k \rho^{i-k} \doublebar{\nabla^i v}_{L^{q}(B(x_0,\rho))}.\end{equation*}
\end{lem}

\begin{lem}\label{lem:Morrey}\textup{(Morrey's inequality).}
Suppose that $1\leq q\leq\infty$, that $k>\pdmn/q$, and that $\nabla^k v\in L^q(B(x_0,\rho))$ for some ball $B(x_0,\rho)\subset\R^\dmn$.

Then $v$ is H\"older continuous in~$B(x_0,\rho)$. Furthermore, $v$ satisfies the local bound
\begin{equation*}\doublebar{v}_{L^\infty(B(x_0,\rho))} \leq C(q,k) \sum_{i=0}^k \rho^{i-k} \doublebar{\nabla^i v}_{L^{q}(B(x_0,\rho))}.\end{equation*}
\end{lem}

The next lemma comes from the book \cite{Gia83}, where it was used for a relatively straightforward proof of Theorem~\ref{thm:Meyers} in the second-order case.

\begin{lem}\textup{(\cite[Chapter~V, Theorem~1.2]{Gia83}).}\label{lem:Giaquinta}
Let $Q\subset\R^\dmn$ be a cube and let $g$ and $f$ be two nonnegative, locally integrable functions defined on~$Q$. Suppose that, for any $x\in B$, we have that
\begin{equation*}\sup_{0<r<\dist(x,\partial Q)/2} \fint_{B(x,r)} g^p
\leq b \biggl(\sup_{0<r} \fint_{B(x,r)} g\biggl)^p + 
\sup_{0<r} \fint_{B(x,r)} f^p\end{equation*}
for some constant~$b>0$ and some $p>1$. Then there is some $\varepsilon>0$ depending only on $b$, $p$ and the dimension $\dmn$, such that if $p<q<p+\varepsilon$ and $f\in L^p(B(x_0,R))$, then 
\begin{equation*}\biggl(\fint_{(1/2)Q} g^q\biggr)^{1/q}
\leq C(b,p,q)\biggl(\fint_Q g^p\biggr)^{1/p}
+C(b,p,q)\biggl(\fint_Q f^q\biggr)^{1/q}
\end{equation*}
where $(1/2)Q$ is the cube concentric to $Q$ with side-length half that of~$Q$.
\end{lem}

The following lemma was established in \cite[Section~9, Lemma~2]{FefS72} in the case of harmonic functions. We must now generalize it. 
\begin{lem}\label{lem:subaverage} 
Let $0<p_0<q\leq\infty$. Let $x_0\in \R^\dmn$ and let $R>0$. 
Suppose that $u\in L^q(B(x_0,R))$ is a function with the property that, whenever $0<\rho<r<R$, we have the bound
\begin{equation*}
\biggl(\int_{B(x_0,\rho)}\abs{u}^q\biggr)^{1/q}
\leq \frac{C_0}{(r-\rho)^{\pdmn/p_0-\pdmn/q}} \biggl(\int_{B(x_0,r)}\abs{u}^{p_0}\biggr)^{1/p_0}
+F
\end{equation*}
for some constants~$C_0$ and~$F$ depending only on $u$.

Then for every $p$ with $0<p\leq p_0$, there is some constant $C(p,q)$, depending only on $p$, $p_0$, $q$ and $C_0$, such that for any such $\rho$ and~$r$,
\begin{align}
\label{eqn:local-bound:step1}
\biggl(\int_{B(x_0,\rho)}\abs{u}^q\biggr)^{1/q}
&\leq
	\frac{C(p,q)}{(r-\rho)^{\pdmn/p-\pdmn/q}} \biggl(\int_{B(x_0,r)}\abs{u}^{p}\biggr)^{1/p}
	+C(p,q)\,F
.\end{align}
\end{lem}

\begin{proof}

Let $\rho=\rho_0<\rho_1<\rho_2<\dots<r$ for some $\rho_k$ to be chosen momentarily, and let $B_k=B(x_0,\rho_k)$. If $0<\tau<1$, then
\begin{align*}
\doublebar{u}_{L^{p_0}({B_k})}
&=
\biggl(\int_{B_k} \abs{u}^{p_0}\biggr)^{1/p_0}
=
\biggl(\int_{B_k} \abs{u}^{\tau p_0} \abs{u}^{(1-\tau) p_0}\biggr)^{1/p_0}
.\end{align*}
If $0<\tau\leq p/p_0$, then $p/\tau p_0\geq 1$ and so we may apply H\"older's inequality to see that
\begin{align*}
\doublebar{u}_{L^{p_0}({B_k})}
&\leq
\doublebar{u}_{L^{p}({B_k})}^\tau
\doublebar{u}_{L^\gamma({B_k})}^{1-\tau}
\end{align*}
where $\gamma$ satisfies $1/p_0=\tau/p+(1-\tau)/\gamma$.
Choose $\tau$ so that $\gamma=q$; observe that this means that $\tau=(p/p_0)(q-p_0)/(q-p)$, and thus if $0<p<p_0<q$ then $\tau$ does satisfy the condition $0<\tau<p/p_0$.

In order for our estimates to scale correctly, we rewrite this estimate as
\begin{align}
\label{eqn:local-bound:step2}
\frac{\doublebar{u}_{L^{p_0}({B_k})}}{(r-\rho)^{\pdmn/p_0}}
&\leq
	\biggl(\frac{ \doublebar{u}_{L^{p}({B_k})} }{(r-\rho)^{\pdmn/p}}\biggr)^\tau
	\biggl(\frac{ \doublebar{u}_{L^q({B_k})} }{(r-\rho)^{\pdmn/q}}\biggr)^{1-\tau}
.\end{align}

By the bound~\eqref{eqn:local-bound:step1},
\begin{align*}
\frac{ \doublebar{u}_{L^q({B_k})} }{(r-\rho)^{\pdmn/q}}
&\leq
	\frac{C(p_0,q)\doublebar{u}_{L^{p_0}(B_{k+1})}} {(\rho_{k+1}-\rho_k)^{\beta}(r-\rho)^{\pdmn/q}}
	+
	\frac{C(p_0)\,F} {(r-\rho)^{\pdmn/q}}	
\end{align*}
where we have set $\beta={d/p_0-d/q}$. Notice $\beta>0$.

Recall that $\rho_0=\rho$. Let $\rho_{k+1}=\rho_k + (r-\rho)(1-\sigma) \sigma^k$ for some constant $0<\sigma<1$ to be chosen momentarily. Notice that $\lim_{k\to\infty} \rho_k=r$. Because $\sigma^{-k\beta}>1>(1-\sigma)^{\beta}$, we have that
\begin{align*}
\frac{ \doublebar{u}_{L^q({B_k})} }{(r-\rho)^{\pdmn/q}}
&\leq
	\frac{C(p_0,q)\,F} {(r-\rho)^{\pdmn/q}}	
	+
	\sigma^{-k\beta}
	\frac{C(p_0,q)\doublebar{u}_{L^{p_0}(B_{k+1})}} {(1-\sigma)^{\beta} (r-\rho)^{\pdmn/p_0}}
\\&\leq
	C(p_0,q,\sigma)
	\sigma^{-k\beta}
	\biggl(
	\frac{F} {(r-\rho)^{\pdmn/q}}	
	+
	\frac{\doublebar{u}_{L^{p_0}(B_{k+1})}} {(r-\rho)^{\pdmn/p_0}}
	\biggr)
.\end{align*}
By the bound \eqref{eqn:local-bound:step2} and Young's inequality, we have that
\begin{align*}
\frac{\doublebar{u}_{L^{p_0}({B_k})}}{(r-\rho)^{\pdmn/p_0}}
&\leq
	\tau C(p_0,q,\sigma) \sigma^{-k\beta(1-\tau)/\tau}
	\frac{ \doublebar{u}_{L^{p}({B_k})} }{(r-\rho)^{\pdmn/p}}
	+
	(1-\tau)\frac{F} {(r-\rho)^{\pdmn/q}}	
	\\&\qquad
	+
	(1-\tau)
	\frac{\doublebar{u}_{L^{p_0}(B_{k+1})}} {(r-\rho)^{\pdmn/p_0}}
.\end{align*}

Applying this bound to $k=0$ and iterating, we have that for any integer~$K\geq 1$,
\begin{align*}
\frac{\doublebar{u}_{L^{p_0}({B_0})}}{(r-\rho)^{\pdmn/p}}
&\leq
	\sum_{k=0}^{K} (1-\tau)^k
	\biggl(
	\tau C(p_0,q,\sigma) \sigma^{-k\beta(1-\tau)/\tau}
	\frac{ \doublebar{u}_{L^{p}({B_k})} }{(r-\rho)^{\pdmn/p}}
	\biggr)
	\\&\qquad+
	\sum_{k=0}^{K} (1-\tau)^k
	\biggl(
	(1-\tau)\frac{F} {(r-\rho)^{\pdmn/q}}	
	\biggr)
	\\&\qquad
	+
	(1-\tau)^{K+1}
	\frac{\doublebar{u}_{L^{p_0}(B_{K+1})}} {(r-\rho)^{\pdmn/p_0}}
.\end{align*}
We want to take the limit as $K\to \infty$. Choose $\sigma$ so that $(1-\tau)<\sigma^{\beta(1-\tau)/\tau}<1$; then the sums converge and we have that
\begin{align*}
\frac{\doublebar{u}_{L^{p_0}({B(x_0,r)})}}{(r-\rho)^{\pdmn/p}}
&\leq
	C(p_0,p,q) 
	\frac{ \doublebar{u}_{L^{p}({B(x_0,r)})} }{(r-\rho)^{\pdmn/p}}
	+C(p_0,p,q)
	\frac{F} {(r-\rho)^{\pdmn/q}}	
.\end{align*}
This completes the proof.
\end{proof}


\begin{proof}[Proof of Theorem~\ref{thm:Meyers}]
We begin with the bound \eqref{eqn:Meyers}.

Let $x_1\in \R^\dmn$ and let $\rho>0$ be such that $B(x_1,2\rho)\subset B(x_0,R)$. By Lemma~\ref{lem:Caccioppoli:1},
\begin{align*}\fint_{B(x_1,\rho)} \abs[big]{\nabla^{m}\vec u}^2
&\leq
\sum_{j=1}^{m}
\frac{C}{\rho^{2j}} \fint_{B(x_1,(3/2)\rho)} \abs[big]{\nabla^{m-j}\vec u}^2
+ C\fint_{B(x_1,(3/2)\rho)} h^2
\end{align*}
where $h(x)=\abs{\mydot F(x)}+\delta^{1/2}\abs{\vec u(x)}$. (Recall that $\vec u=0$ in $B(x_0,R)\setminus\Omega$; we may also take $\mydot F=0$ in $B(x_0,R)\setminus\Omega$.)

If $B(x_1,(3/2)\rho)\subset \Omega$, then we 
normalize $\vec u$ by adding polynomials, so that $\fint_{B(x_1,(3/2)\rho)} \nabla^i\vec u=0$ for all $0\leq i\leq m-1$; if $L\vec u=\Div_m\mydot F$ in all of $B(x_1,(3/2)\rho)$ then the above bound is still valid. We may then apply the Poincar\'e inequality to control the integral of $\nabla^{m-j}\vec u$ by the integral of $\nabla^{m-1}\vec u$.
Thus, 
\begin{equation*}\fint_{B(x_1,\rho)} \abs[big]{\nabla^{m}\vec u}^2
\leq
\frac{C}{\rho^2} \fint_{B(x_1,(3/2)\rho)} \abs[big]{\nabla^{m-1}\vec u}^2
+ C\fint_{B(x_1,(3/2)\rho)} h^2
.\end{equation*}
Now, let $2_1'=2\pdmn/(\dmn+2)$.
By Lemma~\ref{lem:GNS:ball},
\begin{align*}\biggl(\fint_{B(x_1,(3/2)\rho)} \abs[big]{\nabla^{m-1}\vec u}^2\biggr)^{1/2}
&\leq
	C\rho \biggl(\fint_{B(x_1,(3/2)\rho)} \abs[big]{\nabla^{m}\vec u}^{2_1'}\biggr)^{1/2_1'}
	\\&\qquad
	+ C\biggl(\fint_{B(x_1,(3/2)\rho)} \abs[big]{\nabla^{m-1}\vec u}^{2_1'}\biggr)^{1/2_1'}
.\end{align*}
Using the Poincar\'e inequality and the assumption that $\fint_{B(x_1,2\rho)}\nabla^{m-1}\vec u=0$, we may control the second term on the right-hand side by the first; we thus have the bound
\begin{align*}\biggl(\fint_{B(x_1,\rho)} \abs[big]{\nabla^{m}\vec u}^2\biggr)^{1/2}
&\leq
C \biggl(\fint_{B(x_1,(3/2)\rho)} \abs[big]{\nabla^{m}\vec u}^{2_1'}\biggr)^{1/2_1'}
+ C\biggl(\fint_{B(x_1,(3/2)\rho)} h^2\biggr)^{1/2}
.\end{align*}

If $B(x_1,(3/2)\rho)\not\subset\Omega$,
then there is some $x_2\in \partial\Omega\cap B(x_1,(3/2)\rho)$. By our assumption on~$\Omega$, 
\begin{equation*}2^{-\dmn} c_\Omega \rho^\dmn \leq \abs{B(x_2,\rho/2)\setminus\Omega}\leq\abs{B(x_1,2\rho)\setminus\Omega}
.\end{equation*}
Then $\nabla^{m-j}\vec u=0$ in the substantial set ${B(x_1,2\rho)\setminus\Omega}$ for all~$j$. Thus, we may use the Poincar\'e inequality in $B(x_1,2\rho)$ without renormalizing~$\vec u$. Arguing as before we have the bound 
\begin{equation*}\biggl(\fint_{B(x_1,\rho)} \abs[big]{\nabla^{m}\vec u}^2\biggr)^{1/2}
\leq
C \biggl(\fint_{B(x_1,2\rho)} \abs[big]{\nabla^{m}\vec u}^{2_1'}\biggr)^{1/2_1'}
+ C\biggl(\fint_{B(x_1,2\rho)} h^2\biggr)^{1/2}
.\end{equation*}

Observe that $2_1'<2$. Thus we have established a reverse H\"older inequality.
In particular, the bound \eqref{eqn:Meyers} is valid for $R=2r=2\rho$,
for $q=2$ and for $p=2_1'$.

We now use Lemma~\ref{lem:Giaquinta} to improve to $q>2$.
Observe that we may cover $B(x_0,r)$ by a grid of cubes $Q_j$, $1\leq j\leq J$, with side-length $\ell(Q_j)=(R-r)/2c_0$, with pairwise-disjoint interiors. If we choose $c_0$ large enough (depending on the dimension), then $2Q_j\subset B(x_0,R)$ for all~$j$. We then have that, for any~$p$,
\begin{equation*}\int_{B(x_0,r)} \abs{\nabla^m\vec u}^p
\leq \sum_{j=1}^J \int_{Q_j}  \abs{\nabla^m\vec u}^p.\end{equation*}

Fix some $j$.
Let $g(x)=\abs{\nabla^m u(x)}^{2_1'}$, and let $f(x)=h(x)^{2_1'}$. Let $p=2/2_1'$; notice $p>1$.

If $x_1\in Q_j$, and if $0<\rho<\dist(x_1,\partial Q_j)/2$, then
\begin{align*}
\fint_{B(x_1,\rho)} g^p
&=
	\fint_{B(x_1,\rho)} \abs{\nabla^{m} u(x)}^2
\\&\leq 
	C \biggl(\fint_{B(x_1,2\rho)} \abs[big]{\nabla^{m}\vec u}^{2_1'}\biggr)^{2/2_1'}
	+ C\fint_{B(x_1,2\rho)} h^2
\\&=
	C \biggl(\fint_{B(x_1,2\rho)} g\biggr)^p
	+ C\fint_{B(x_1,2\rho)} f^p
.\end{align*}
Thus Lemma~\ref{lem:Giaquinta} applies, and so there is some $q^+>2$ such that
\begin{align*}
\biggl(\fint_{Q_j}\abs{\nabla^m \vec u}^q\biggr)^{1/q}
&\leq 
	C(q)\biggl(\fint_{2Q_j}\abs{\nabla^m \vec u}^2\biggr)^{1/2} 
	+ C(q)\biggl(\fint_{2Q_j}h^q\biggr)^{1/q}
\end{align*}
for all $q$ with $2<q<p^+$.
Thus, 
\begin{align*}
\int_{B(x_0,r)} \abs{\nabla^m\vec u}^q
&\leq \sum_{j=1}^J \int_{Q_j}  \abs{\nabla^m\vec u}^q
\\&\leq
	\sum_{j=1}^J
	\frac{C(q)}{\ell(Q_j)^{\pdmn q/2-\pdmn}}
	\biggl( \int_{2Q_j}  \abs{\nabla^m\vec u}^2\biggr)^{q/2}
	+
	C(q)\sum_{j=1}^J \int_{2Q_j}  h^q
.\end{align*}
Recall that $\ell(Q_j) = (R-r)/2c_0$.
Observe that almost every $x\in B(x_0,R)$ is in at most $2^\dmn$ of the cubes $2Q_j$; thus,
\begin{align*}
\int_{B(x_0,r)} \abs{\nabla^m\vec u}^q
&\leq 
	\frac{C(q)}{(R-r)^{\pdmn q/2-\pdmn}}
	\biggl( 
	\int_{B(x_0,R)}  \abs{\nabla^m\vec u}^2\biggr)^{q/2}
	+
	C(q)\int_{B(x_0,R)}  h^q
\end{align*}
as desired.

Applying Lemma~\ref{lem:subaverage}, we see that we may replace the exponent~$2$ by any exponent $p>0$; this completes the proof of the bound \eqref{eqn:Meyers}.

Now, suppose that $0<k<\pdmn/2$. We wish to prove the bound \eqref{eqn:Meyers:lower}.
We apply Lemma~\ref{lem:GNS:ball} to $v=\nabla^{m-k}\vec u$. This gives us the bound 
\begin{align*}
\biggl(\int_{B(x_1,\rho)}\abs{\nabla^{m-k} \vec u}^{q_k}\biggr)^{1/q_k}
&\leq
	C\sum_{i=0}^k \rho^{-i}
	\biggl(\int_{B(x_1,\rho)}\abs{\nabla^{m-i} \vec u}^{q}\biggr)^{1/q}
.\end{align*}
We have that
\begin{align*}
\biggl(\int_{B(x_1,\rho)}\abs{\nabla^{m-i} \vec u}^{q}\biggr)^{1/q}
&\leq
	\biggl(\int_{B(x_1,\rho)}\abs{\nabla^{m-i} \vec u-{\textstyle\fint_{B(x_1,\rho)} \nabla^{m-i} \vec u}}^{q}\biggr)^{1/q}
	\\&\qquad
	+
	C\rho^{\dmn/q}\abs{\textstyle\fint_{B(x_1,\rho)} \nabla^{m-i} \vec u}
\end{align*}
and so by the Poincar\'e inequality
\begin{align*}
\biggl(\int_{B(x_1,\rho)}\abs{\nabla^{m-i} \vec u}^{q}\biggr)^{1/q}
&\leq
	C\rho\biggl(\int_{B(x_1,\rho)}\abs{\nabla^{m-i+1} \vec u}^{q}\biggr)^{1/q}
	\\&\qquad
	+
	C\rho^{\pdmn/q-\pdmn}\int_{B(x_1,\rho)} \abs{\nabla^{m-i} \vec u}
.\end{align*}
Iterating, we see that
\begin{align*}
\biggl(\int_{B(x_1,\rho)}\abs{\nabla^{m-k} \vec u}^{q_k}\biggr)^{1/q_k}
&\leq
	C\biggl(\int_{B(x_1,\rho)}\abs{\nabla^{m} \vec u}^{q}\biggr)^{1/q}
	\\&\qquad
	+C\sum_{i=0}^k \rho^{-i+\pdmn/q-\pdmn}\int_{B(x_1,\rho)} \abs{\nabla^{m-i} \vec u}
.\end{align*}
Applying the known results for $\nabla^m\vec u$ and Corollary~\ref{cor:Caccioppoli:interior} or~\ref{cor:Caccioppoli:Dirichlet}, we see that
\begin{align*}
\biggl(\int_{B(x_1,\rho)}\abs{\nabla^{m-k} \vec u}^{q_k}\biggr)^{1/q_k}
&\leq
	\frac{C(q)}{\rho^{\pdmn/2-\pdmn/q+m}}
	\biggl(\int_{B(x_1,(3/2)\rho)} \abs{\vec u}^2\biggr)^{1/2} 
	\\&\qquad
	+ C(q) \biggl(\int_{B(x_1,(3/2)\rho)}h^q\biggr)^{1/q}
.\end{align*}

As before, we either normalize $\vec u$ in $B(x_1,(3/2)\rho)$ by adding polynomials of degree $m-k-1$ or observe that $\vec u$ and all its derivatives are zero on a substantial subset of $B(x_1,2\rho)$; in either case we may use the Poincar\'e inequality to control $\vec u$ by $\nabla^{m-k}\vec u$. This yields the bound 
\begin{align*}
\biggl(\int_{B(x_1,\rho)}\abs{\nabla^{m-k} \vec u}^{q_k}\biggr)^{1/q_k}
&\leq
	\frac{C(q)}{\rho^{\pdmn/2-\pdmn/q+k}}
	\biggl(\int_{B(x_1,2\rho)} \abs{\nabla^{m-k} \vec u}^2\biggr)^{1/2} 
	\\&\qquad
	+ C(q) \biggl(\int_{B(x_1,2\rho)}h^q\biggr)^{1/q}
.\end{align*}
By H\"older's inequality we may replace the exponent $2$ by the exponent $p_k$ provided $p_k\geq 2$.
Using standard covering lemmas, if $q_k\geq \max(p_k,q)$ then we may improve to the estimate
\begin{align*}
\biggl(\int_{B(x_1,\rho)}\abs{\nabla^{m-k} \vec u}^{q_k}\biggr)^{1/q_k}
&\leq
	\frac{C(q)}{(r-\rho)^{\pdmn/p_k-\pdmn/q+k}}
	\biggl(\int_{B(x_0,r)} \abs{\nabla^{m-k} \vec u}^{p_k}\biggr)^{1/{p_k}} 
	\\&\qquad
	+ C(q) \biggl(\int_{B(x_0,r)}h^q\biggr)^{1/q}
.\end{align*}
By Lemma~\ref{lem:subaverage} this inequality is still valid for $0<{p_k}<2$.

Identical arguments, using Lemma~\ref{lem:Morrey} in place of Lemma~\ref{lem:GNS:ball}, establish the bound \eqref{eqn:Meyers:lowest} on $\sup \abs{\nabla^{m-k}\vec u}$ in the case $k>\pdmn/q$.
\end{proof}

In some domains we may also prove a boundary reverse H\"older estimate in the Neumann case.
\begin{thm} 
\label{thm:Meyers:Neumann}
Let $\Omega$ be a Lipschitz graph domain, that is, a domain of the form 
\begin{equation*}\Omega=\{(x',t):x'\in\R^{\dmnMinusOne},\>t>\varphi(x')\}\end{equation*}
for some function $\varphi:\R^\dmnMinusOne\mapsto\R$ with $\doublebar{\nabla\varphi}_{L^\infty(\R^\dmnMinusOne)}=M<\infty$.

Let $L$ be an operator of order~$2m$ that satisfies the bound~\eqref{eqn:elliptic:bounded:2} and the bound \eqref{eqn:elliptic:local} in~$\Omega$.

Then there is some number $p^+=p^+_L>2$ depending only on the standard constants and the number $M=\doublebar{\nabla\varphi}_{L^\infty(\R^\dmnMinusOne)}$ such that the following statement is true.

Let $x_0\in\partial\Omega$ and let $R>0$. Suppose that 
$\vec u\in \dot W^2_m(B(x_0,R))$, that $\mydot F\in L^2(B(x_0,R))$, and that 
\begin{equation*}\bigl\langle \nabla^m\vec\varphi, \mat A\nabla^m\vec u\bigr\rangle_\Omega = \bigl\langle \nabla^m\vec\varphi, \mydot F\bigr\rangle_\Omega \end{equation*}
for all smooth functions $\vec\varphi$ supported in $B(x_0,R)$.

Then
\begin{align}
\label{eqn:Meyers:Neumann}
\biggl(\int_{B(x_0,r)\cap\Omega}\abs{\nabla^m \vec u}^q\biggr)^{1/q}
&\leq 
	\frac{C(M,p,q)}{(R-r)^{\pdmn/p-\pdmn/q}}
	\biggl(\int_{B(x_0,R)\cap\Omega}\abs{\nabla^m \vec u}^p\biggr)^{1/p} 
	\\&\qquad\nonumber
	+ C(M,p,q) \biggl(\int_{B(x_0,R)\cap\Omega}\abs{\mydot F}^q+\delta^{q/2}\abs{\vec u}^q\biggr)^{1/q}
\end{align}
for some constant $C(M,p,q)$ depending only on $p$, $q$, $M$ and the standard parameters.
\end{thm}

\begin{proof}
If $x_1=(x_1',t_1)\in\R^\dmn$ and $\rho>0$, then let $Q(x_1,\rho)$ be the Lipschitz cylinder
\begin{equation*}Q(x_1,\rho)=\{(x',t):\abs{x'-x_1'}<\rho,\> \varphi(x')+t_1-\rho<t<\varphi(x')+t_1+\rho\}.\end{equation*}
Using either covering lemmas or a bilipschitz change of variables, we see that many results stated in terms of balls are valid in Lipschitz cylinders. In particular, Lemma~\ref{lem:Caccioppoli:Neumann}, the Poincar\'e inequality, and the first-order Gagliardo-Nirenberg-Sobolev inequality 
\begin{equation*}\doublebar{v}_{L^{q_1}(Q(x_0,\rho))}\leq C \doublebar{\nabla v}_{L^q(Q(x_0,\rho))} +C\rho\doublebar{v}_{L^q(Q(x_0,\rho))},\end{equation*}
Lemma~\ref{lem:Giaquinta}, and Lemma~\ref{lem:subaverage}
are valid in Lipschitz cylinders. 

We now proceed much as in the proof of the estimate~\eqref{eqn:Meyers} of Theorem~\ref{thm:Meyers}.
Let $x_1\in \R^\dmn$ and let $\rho>0$ be such that $Q(x_1,2\rho)\subset B(x_0,R)$. By Lemma~\ref{lem:Caccioppoli:Neumann},
\begin{align*}\biggl(\fint_{Q(x_1,\rho)} 
\1_\Omega\abs[big]{\nabla^{m}\vec u}^2\biggr)^{1/2}
&\leq
\sum_{j=1}^{m}
\frac{C}{\rho^{2j}} \biggl(\fint_{Q(x_1,(3/2)\rho)} \1_\Omega\abs[big]{\nabla^{m-j}\vec u}^2\biggr)^{1/2}
\\&\qquad
+ C\biggl(\fint_{Q(x_1,(3/2)\rho)} h^2\biggr)^{1/2}
\end{align*}
where $h(x)=\abs{\mydot F(x)}+\delta^{1/2}\abs{\vec u(x)}$ in $\Omega$ and is zero outside~$\Omega$.

Notice that we may normalize $\vec u$ by adding polynomials, regardless of whether $Q(x_1,(3/2)\rho)$ is contained in~$\Omega$.
If $Q(x_1,(3/2)\rho)\subset \Omega$, then may establish the reverse H\"older inequality 
\begin{align*}\biggl(\fint_{Q(x_1,\rho)} \1_\Omega\abs[big]{\nabla^{m}\vec u}^2\biggr)^{1/2}
&\leq
C \biggl(\fint_{Q(x_1,(3/2)\rho)} \1_\Omega\abs[big]{\nabla^{m}\vec u}^{2_1'}\biggr)^{1/2_1'}
\\&\qquad
+ C\biggl(\fint_{Q(x_1,(3/2)\rho)} h^2\biggr)^{1/2}
\end{align*}
as in the proof of Theorem~\ref{thm:Meyers}.
If $Q(x_1,(3/2)\rho)\not\subset\Omega$, either $Q(x_1,(3/2)\rho)\cap\Omega=\emptyset$ and so this reverse H\"older inequality is trivially true, or $Q(x_1,2\rho)\cap\Omega$ is substantial. Specifically, in this final case there exists some $c$ with $4/3<c<8$ such that the map $(x,t)\mapsto (x,ct)$ sends $Q(x_1,2\rho)\cap\Omega$ to a Lipschitz cylinder. Thus, Lemma~\ref{lem:GNS:ball} and the Poincar\'e inequality are valid in $Q(x_1,2\rho)\cap\Omega$ with constants independent of $x_1$ and~$\rho$, and so we see that
\begin{align*}
\biggl(\fint_{Q(x_1,\rho)} \1_\Omega\abs[big]{\nabla^{m}\vec u}^2\biggr)^{1/2}
&\leq
C \biggl(\fint_{Q(x_1,2\rho)} \1_\Omega\abs[big]{\nabla^{m}\vec u}^{2_1'}\biggr)^{1/2_1'}
+ C\biggl(\fint_{Q(x_1,2\rho)} h^2\biggr)^{1/2}
.\end{align*}
This establishes a reverse H\"older inequality with $q=2$ and $p=2_1'$; as in the proof of Theorem~\ref{thm:Meyers}, we may use Lemmas~\ref{lem:Giaquinta} and~\ref{lem:subaverage} and covering lemmas to improve to arbitrary $p$, $q$ and to return to balls of radii $r$ and~$R$.
\end{proof}

\section{The fundamental solution}
\label{sec:fundamental}

In this section we will construct the fundamental solution for elliptic systems of arbitrary order $2m\geq 2$ in dimension $\dmn\geq 2$.
As in \cite{GruW82,HofK07}, we will construct the fundamental solution as the kernel of the solution operator to the equation $L\vec u=\Div_m\mydot F$.

Specifically, in Section~\ref{sec:newton} we will construct this solution operator using the Lax-Milgram lemma and will discuss its adjoint. In Section~\ref{sec:fundamental:high} we will construct a preliminary version of the fundamental solution in the case of operators of high order. In Section~\ref{sec:fundamental:further} we will refine our construction to produce some desirable additional properties, and finally in Section~\ref{sec:fundamental:low} we will extend these results to operators of arbitrary even order. A summary of the principal results concerning the fundamental solution is collected at the beginning of Section~\ref{sec:fundamental:low}.

\subsection{The Newton potential}
\label{sec:newton}

In this section we will construct the Newton potential, that is, the operator whose kernel is the fundamental solution. 
The Newton potential $\vec u=\vec \Pi^L\mydot F$ is defined as the solution to $L\vec u=\Div_m\mydot F$ in $\R^\dmn$. If $\mydot F\in L^2(\R^\dmn)$, then we may construct $\vec\Pi^L\mydot F$ as follows.

Recall the (complex) Lax-Milgram lemma:
\begin{thm}\textup{({\cite[Theorem~2.1]{Bab70}}).}
\label{thm:lax-milgram}
Let $H_1$ and $H_2$ be two Hilbert spaces, and let $B$ be a bounded bilinear form on $H_1\times H_2$ that is coercive in the sense that
	\begin{equation*}\sup_{w\in H_1\setminus\{0\}} \frac{\abs{B(w,v)}}{\doublebar{w}_{H_1}}\geq \lambda \doublebar{v}_{H_2},\quad
	\sup_{w\in H_2\setminus\{0\}} \frac{\abs{B(u,w)}}{\doublebar{w}_{H_2}}\geq \lambda \doublebar{u}_{H_1}\end{equation*}
	for every $u\in {H_1}$, $v\in {H_2}$, for some fixed $\lambda>0$. Then for every linear functional $T$ defined on ${H_2}$ there is a unique $u_T\in {H_1}$ such that $B(v,u_T)=\overline{T(v)}$. Furthermore, $\doublebar{u_T}_{H_1}\leq \frac{1}{\lambda}\doublebar{T}_{H_1\mapsto H_2}$.
\end{thm}

Let $L$ be an operator of order $2m$ that is elliptic in the sense that the coefficients satisfy the conditions \eqref{eqn:elliptic:bounded:2} and~\eqref{eqn:elliptic}.
Suppose that $\mydot F=\{F_{j,\alpha}:1\leq j\leq N,\abs{\alpha}=m\}$ is an array of functions all lying in $L^2(\R^\dmn\mapsto \C)$. 
Then $T_{\mydot F}(\vec v)=\bigl\langle \mydot F, \nabla^m \vec v\bigr\rangle_{\R^\dmn}$ is a bounded linear operator on the Hilbert space $\dot W^2_m(\R^\dmn)$. We choose $B(\vec w,\vec v) = \bigl\langle \nabla^m \vec w, \mat A\nabla^m \vec v\bigr\rangle_{\R^\dmn}$; by our ellipticity conditions \eqref{eqn:elliptic:bounded:2} and \eqref{eqn:elliptic}, $B$ is bounded and coercive on $\dot W^2_m(\R^\dmn)$. Let $\vec\Pi^L\mydot F$ be the element $u_T$ of $\dot W^2_m(\R^\dmn)$ given by the Lax-Milgram lemma. Then
\begin{equation}
\label{dfn:newton}
\bigl\langle\nabla^m\vec\varphi, \mat A\nabla^m(\vec\Pi^L\mydot F) \bigr\rangle_{\R^\dmn} 
=
\bigl\langle\nabla^m\vec\varphi, \mydot F \bigr\rangle_{\R^\dmn}
\end{equation}
for all $\vec\varphi\in \dot W^2_m(\R^\dmn\mapsto\C^N)$.

We will need some properties of the Newton potential $\vec\Pi^L$. First, by the uniqueness of solutions provided by the Lax-Milgram lemma, $\vec\Pi^L$ is a well-defined operator; furthermore, $\vec \Pi^L$ is linear and bounded $L^2(\R^\dmn)\mapsto \dot W^2_m(\R^\dmn)$.

Next, observe that if $\vec\Phi \in \dot W^2_m(\R^\dmn\mapsto\C^N)$, then by uniqueness of solutions to $L\vec u=\Div_m\mydot F$,
\begin{equation}\label{eqn:newton:identity}
\vec \Pi^L(\mat A\nabla^m\vec\Phi)=\vec\Phi\end{equation}
as $\dot W^2_m(\R^\dmn\mapsto\C^N)$-functions, that is, up to adding polynomials of order $m-1$.

Next, we wish to show that the adjoint $(\nabla^m\vec\Pi^L)^*$ to the operator $\nabla^m\vec\Pi^L$ is $\nabla^m\vec\Pi^{L^*}$. To prove this we will need the following elementary result; this will let us identify vector fields that arise as $m$th-order gradients.

\begin{lem}
\label{lem:curlfree}
Let $\begin{pmatrix} f_\alpha\end{pmatrix}_{\abs{\alpha}=m}$ be a set of functions in $L^1_{loc}(\Omega)$, where $\Omega$ is a simply connected domain. Suppose that whenever $\alpha+\vec e_k=\beta+\vec e_j$, we have that 
\begin{equation*}\bigl\langle \partial_j\varphi,f_\beta\bigr\rangle_\Omega
=\bigl\langle\partial_k\varphi, f_\alpha\bigr\rangle_\Omega\end{equation*}
for all $\varphi$ smooth and compactly supported in~$\Omega$.

Then there is some function $f\in \dot W^1_{m,loc}(\Omega)$ such that $f_\alpha=\partial^\alpha f$ for all~$\alpha$.
\end{lem}

\begin{proof} 
If $m=1$ and the functions $f_\alpha$ are $C^1$, then this lemma is merely the classical result that irrotational vector fields may be written as gradients. We begin by generalizing to the case $m=1$ and the case $f_\alpha\in L^1_{loc}(\Omega)$.
We let $f_j=f_{\vec e_j}$.


Let $\eta$ be a smooth, nonnegative function supported in $B(0,1)$ with $\int \eta = 1$, and let $\eta_\varepsilon(x)=\varepsilon^{-\pdmn}\eta(x/\varepsilon)$. Let $f^\varepsilon_j=f_j*\eta_\varepsilon$, so that $f^\varepsilon_j$ is smooth. By assumption, $\partial_k f^\varepsilon_j(x) = \partial^j f^\varepsilon_k(x)$ provided $\varepsilon < \dist(x,\Omega^C)$.
Let $B$ be a ball  with $\overline B\subset\Omega$, and assume that $\varepsilon<\dist(B,\Omega^C)/2$. Then there is some function $f^\varepsilon$ such that $\partial^j f^\varepsilon = f^\varepsilon_j$ in~$B$.

Now renormalize $f^\varepsilon$ so that $\int_B f^\varepsilon=0$. By Lemma~\ref{lem:GNS:ball}, because $\nabla f^\varepsilon \in L^1(B)$, we have that $f^\varepsilon\in L^p(B)$, uniformly in~$\varepsilon$, for some $p>1$. Since $L^p(B)$ is weakly sequentially compact, we have that some subsequence $f^{\varepsilon_i}$ has a weak limit~$f$.

If $\varphi$ is smooth and supported in~$B$, then
\begin{equation*}\bigl\langle \partial^j \varphi, f\bigr\rangle_V 
= \lim_{i\to\infty}\bigl\langle \partial^j \varphi, f^{\varepsilon_i} \bigr\rangle_V 
= -\lim_{i\to\infty}\bigl\langle \varphi, f^{\varepsilon_i}_j \bigr\rangle_V 
= -\bigl\langle \varphi, f_j\bigr\rangle_V 
\end{equation*}
and so $f_j$ is the weak derivative of~$f$ in the $j$th direction for all $1\leq j\leq \dmn$.

We may cover any compact subset $\overline V\subset \Omega$ by such balls~$B$; renormalizing $f$ again, so as to be defined compatibly on different balls, we see that we may extend $f$ to a function in $L^1_{loc}(\Omega)$.

Now we work by induction. Suppose that the theorem is true for $m=1$ and for $m=M-1$. We wish to show that the theorem is true for $m=M$ as well.

Fix some $\gamma$ with $\abs{\gamma}=M-1$, and let $f_j=f_{\gamma+\vec e_j}$. By assumption \begin{equation*}
\bigl\langle \partial_k  \varphi, f_j\bigr\rangle_\Omega
=\bigl\langle \partial_k  \varphi, f_{\gamma+\vec e_j}\bigr\rangle_\Omega
=\bigl\langle \partial_j  \varphi,  f_{\gamma+\vec e_k}\bigr\rangle_\Omega
=\bigl\langle \partial_j  \varphi,  f_k\bigr\rangle_\Omega
\end{equation*}
for all appropriate test functions~$\varphi$.

Because the theorem is valid for $m=1$, there is some $f=f_\gamma\in \dot W^1_{1,loc}(\Omega)$ such that $\partial_jf_\gamma=f_{\gamma+\vec e_j}$ in the weak sense.

If $\abs{\gamma}=\abs{\delta}=M-1$, and $\gamma+\vec e_j=\delta+\vec e_k$, then
\begin{equation*}
\bigl\langle \partial_j  \varphi, f_\gamma\bigr\rangle_\Omega
= -\bigl\langle  \varphi,f_{\gamma+\vec e_j}\bigr\rangle_\Omega
= -\bigl\langle  \varphi,f_{\delta+\vec e_k}\bigr\rangle_\Omega
= \bigl\langle \partial_k  \varphi,f_\delta\bigr\rangle_\Omega
\end{equation*}
and so the array $\begin{pmatrix} f_\gamma\end{pmatrix}_{\abs{\gamma}=M-1}$ satisfies the conditions of the theorem with $m=M-1$. Because the theorem is true for $m=M-1$, we have that there is some $f\in \dot W^{1}_{M-1,loc}(\Omega)$ such that $f_\gamma=\partial^\gamma f$ for all $\abs{\gamma}=M-1$; because $\partial_k f_\gamma=f_{\gamma+\vec e_k}$ we have that $f_\alpha=\partial^\alpha f$ for all $\abs{\alpha}=m$, and so the theorem is true for $m=M$ as well. This completes the proof.
\end{proof}

We now consider the adjoint operator to the Newton potential.

\begin{lem} \label{lem:newton:adjoint}
The adjoint $(\nabla^m\vec\Pi^L)^*$ to the operator $\nabla^m\vec\Pi^L$ is $\nabla^m\vec\Pi^{L^*}$.
\end{lem}
\begin{proof}
Observe that $\nabla^m\vec\Pi^L$ is bounded on $L^2(\R^\dmn)$ and so $(\nabla^m\vec\Pi^L)^*$ is as well; that is, $(\nabla^m\vec\Pi^L)^*\mydot F$ is an element of $L^2(\R^\dmn)$.
We first show that it is an element of the subspace of gradients of $\dot W^2_m(\R^\dmn)$-functions, that is, that there is some function $\vec u\in \dot W^2_m(\R^\dmn)$ such that $(\nabla^m\vec\Pi^L)^*\mydot F=\nabla^m \vec u$.

By Lemma~\ref{lem:curlfree}, it suffices to show that if $1\leq i\leq N$, if $\varphi$ is smooth and compactly supported in~$\Omega$, and if $\alpha+\vec e_k=\beta+\vec e_j$, then 
\begin{equation*}\bigl\langle \partial_j\varphi \,\mydot e_{i,\beta}, (\nabla^m\vec\Pi^L)^*\mydot F\bigr\rangle_\Omega
=\bigl\langle\partial_k\varphi \, \mydot e_{i,\alpha}, (\nabla^m\vec\Pi^L)^*\mydot F\bigr\rangle_\Omega.
\end{equation*}
That is, we seek to show that
\begin{equation*}
\bigl\langle \nabla^m\vec\Pi^L (\partial_j\varphi\, \mydot e_{i,\beta}-\partial_k\varphi\,\mydot e_{i,\alpha}),\mydot F\bigr\rangle_\Omega=0.
\end{equation*}
But $\bigl\langle\nabla^m\vec\eta,\partial_j\varphi \,\mydot e_{i,\beta}-\partial_k\varphi\,\mydot e_{i,\alpha}\bigr\rangle_{\R^\dmn}=0$ for all $\vec\eta$ smooth and compactly supported, and so $\vec\Pi^L (\partial_j\varphi\, \mydot e_{i,\beta}-\partial_k\varphi\,\mydot e_{i,\alpha})=0$.

Let $\vec u$ satisfy $\nabla^m\vec u=(\nabla^m\vec\Pi^L)^*\mydot F$. We now show that $\vec u=\vec\Pi^{L^*}\mydot F$. 
Choose some  $\vec \varphi$ smooth and compactly supported in $\R^\dmn$.
Then
\begin{align*}
\bigl\langle \nabla^m\vec \varphi, \mat A^*\nabla^m \vec u\bigr\rangle_{\R^\dmn} 
&=
	\bigl\langle \mat A\nabla^m\vec \varphi, (\nabla^m\vec\Pi^L)^*\mydot F \bigr\rangle_{\R^\dmn} 
\\&=
	\bigl\langle \nabla^m\vec\Pi^L(\mat A\nabla^m\vec \varphi), \mydot F \bigr\rangle_{\R^\dmn} 
.\end{align*}
By formula~\eqref{eqn:newton:identity}, we have that $\nabla^m\vec\Pi^L(\mat A\nabla^m\vec \varphi)=\nabla^m\vec\varphi$.
Thus
\begin{align*}
\bigl\langle \nabla^m\vec \varphi, \mat A^*\nabla^m \vec u\bigr\rangle_{\R^\dmn} 
=
	\bigl\langle \nabla^m\vec \varphi, \mydot F \bigr\rangle_{\R^\dmn} 
\end{align*}
for all $\vec\varphi$ smooth and compactly supported. Because $\vec\Pi^{L^*}\mydot F$ is the unique element of $\dot W^2_m(\R^\dmn)$ with this property, we must have that $\vec u = \vec\Pi^{L^*}\mydot F$ and the proof is complete.
\end{proof}

We conclude this section by showing that the Newton potential is bounded on a range of $L^p$ spaces.
\begin{lem}\label{lem:Newton:bounded} Let $L$ be an operator of order $2m$ that satisfies the bounds \eqref{eqn:elliptic:bounded:2} and \eqref{eqn:elliptic}, and let $p_L^+$ be as in Theorem~\ref{thm:Meyers}. Let $1/p_L^++1/p_L^-=1$. If $p^-_{L^*}<p<p^+_L$, then $\vec\Pi^L$ extends to an operator that is bounded $L^p(\R^\dmn)\mapsto\dot W^p_m(\R^\dmn)$.
\end{lem}

\begin{proof}
Suppose first that $2<p<p^+_L$. Let $\mydot F\in L^2(\R^\dmn)\cap L^p(\R^\dmn)$ and let $\vec u=\vec \Pi^L\mydot F$. By Theorem~\ref{thm:Meyers},
\begin{align*}\biggl(\int_{B(x_0,r)} \abs{\nabla^m\vec u}^p\biggr)^{1/p}
&\leq \frac{C(p)}{r^{\pdmn/2-\pdmn/p}}
\biggl(\int_{B(x_0,2r)} \abs{\nabla^m\vec u}^2\biggr)^{1/2}
+ C(p)\biggl(\int_{B(x_0,2r)} \abs{\mydot F}^p\biggr)^{1/p}
.\end{align*}
By taking the limit as $r\to\infty$, we see that $\doublebar{\nabla^m\vec\Pi^L\mydot F}_{L^p(\R^\dmn)}\leq C(p)\doublebar{\mydot F}_{L^p(\R^\dmn)}$, and so $\vec\Pi^L$ extends to an operator that is bounded $L^p(\R^\dmn)\mapsto\dot W^p_m(\R^\dmn)$.

By a similar argument $\nabla^m\vec\Pi^{L^*}$ is bounded $L^{p'}(\R^\dmn)\mapsto L^{p'}(\R^\dmn)$ for all $2<p'<p^+_{L^*}$; thus by duality $\nabla^m\vec\Pi^{L}$ is bounded $L^{p}(\R^\dmn)\mapsto L^{p}(\R^\dmn)$ for all $p^-_{L^*}<p<2$, as desired.
\end{proof}

\subsection{The fundamental solution for operators of high order}
\label{sec:fundamental:high}

This section will be devoted to the proof of the following theorem.

\begin{thm}\label{thm:fundamental:high}
Let $L$ be an operator of order $2m>\dmn$ that satisfies the bounds \eqref{eqn:elliptic:bounded:2} and \eqref{eqn:elliptic}. 
For each $z_0\in\R^\dmn$ and each $r>0$, there is an array of functions $E^L_{j,k,z_0,r}(x,y)$ with the following properties.

First, if $x\in\R^\dmn$ and $\abs{\beta}=m$, then $f(y)=\partial_y^\beta E^L_{j,k,z_0,r}(x,y)$ lies in $L^2(\R^\dmn)$, and if
$\mydot F\in L^2(\R^\dmn)$, then for all $1\leq j\leq N$, we have that
\begin{equation}
\label{eqn:fundamental}
\Pi^L_{j}\mydot F(x) 
	= \sum_{k=1}^N \sum_{\abs{\beta}=m} \int_{\R^\dmn} 	\partial_y^\beta E^L_{j,k,z_0,r}(x,y)\,F_{k,\beta}(y)\,dy
\end{equation}as $\dot W^m_2(\R^\dmn)$-functions, that is, up to adding polynomials of order $m-1$.

Next, for any $x_0$ and~$y_0$, we have the bounds
\begin{equation}\label{eqn:fundamental:bound}
\int_{B(x_0,r)} \int_{B(y_0,r)} \abs{\nabla_x^m \nabla_y^m E^L_{j,k,z_0,r}(x,y)}^2\,dy\,dx
\leq C,\qquad r=\abs{x_0-y_0}/3.
\end{equation}
If $1\leq j\leq N$, $1\leq k\leq N$, and if $\alpha$, $\beta$ are multiindices with $\abs{\alpha}=\abs{\beta}=m$, then
\begin{equation}\label{eqn:fundamental:symmetric}
\partial_x^\alpha \partial_y^\beta E^{L^*}_{j,k,z_0,r}(x,y)
= \overline{\partial_y^\beta \partial_x^\alpha E^{L}_{j,k,z_0,r}(y,x)}.
\end{equation}

Finally, if $\abs{x_0-z_0}=\abs{y_0-z_0}=\abs{x_0-y_0}=3r$, then we have the bounds
\begin{equation}\label{eqn:fundamental:bound:m-1}
\int_{B(x_0,r)} \int_{B(y_0,r)} \abs{\nabla_x^{m-q} \nabla_y^{m-s} E^L_{j,k,z_0,r}(x,y)}^2\,dy\,dx
\leq Cr^{2q+2s}
\end{equation}
whenever $0\leq q\leq m$ and $0\leq s\leq m$.
\end{thm}

By uniqueness of the Newton potential $\vec\Pi^L\mydot F$ in $\dot W^2_m(\R^\dmn)$, the array of highest-order derivatives ${\nabla_x^m \nabla_y^m E^L_{j,k}(x,y)}$ is unique; however, there are many possible normalizations of the lower-order derivatives ${\nabla_x^{m-q} \nabla_y^{m-s} E^L_{j,k}(x,y)}$. In Section~\ref{sec:fundamental:further} we will discuss some natural normalization conditions. In Section~\ref{sec:fundamental:low} we will extend this theorem to operators of order $2m\leq \dmn$.

We will now prove Theorem~\ref{thm:fundamental:high}.
We begin by constructing a fundamental solution $\mat E^L_{j,k}(x,y)$. For our preliminary argument, we will need $\vec\Pi^L\mydot F(x)$ to be well-defined for any specified~$x$; that is, we will need to assume that $\vec \Pi^L\mydot F$ is always continuous. Recall that by Lemma~\ref{lem:Morrey}, if $\nabla^m\vec\Pi^L\mydot F\in L^2(\R^\dmn)$ and $m>\pdmn/2$ then $\vec\Pi^L\mydot F$ is continuous. It is for this reason that we begin with operators of order $2m>\dmn$.

Recall that even if $\vec\Pi^L\mydot F$ is continuous, it is still defined only up to adding polynomials of order $m-1$. We will fix a normalization of $\vec\Pi^L\mydot F$ as follows. 
Choose some points $h_1$, $h_2$, \dots $h_{q}\in\R^\dmn$ with $\abs{h_i}=1$, where $q$ is the number of multiindices $\gamma$ with $\abs{\gamma} \leq  m-1$. If the $h_i$s are chosen appropriately, then for any numbers~$a_i$, there is a unique polynomial $P(x)=\sum_{\abs{\gamma}\leq m-1} p_\gamma\, x^\gamma$, of order at most $m-1$, such that $P(h_i)=a_i$ for all $1\leq i\leq p$. Furthermore, there is some constant $H$ depending only on our choice of $h_i$ such that the bound $\abs{p_\gamma}\leq H  \sup_i\abs{a_i}$ is valid.

Now, choose some $z_0\in\R^\dmn$ and some $r>0$.
We fix an additive normalization of $\vec \Pi^L=\vec \Pi^L_{z_0,r}$ by requiring $\vec \Pi^L_{z_0,r}\mydot F(z_0+r\,h_i)=0$ for all $1\leq i\leq q$.

Let $x\in\R^\dmn$. Define $\vec S_x\mydot F= \vec\Pi^L_{z_0,r}\mydot F(x)$. Then $\vec S_x$ is a linear operator. We will use the Riesz representation theorem to construct the fundamental solution as the kernel of $\vec S_x$; to do this, we will need to establish boundedness of~$\vec S_x$.

We will use the following lemma with $u(x)=\vec\Pi^L\mydot F(x)=\vec S_x\mydot F$. 

\begin{lem}\label{lem:high:Morrey:bound}
Let $u$ be a function such that $\nabla^m u\in L^2(\R^\dmn)$ and such that $u(z_0+r\,h_i)=0$ for all $1\leq i\leq q$.

Then 
\begin{equation*}\abs{u(x)}\leq C\biggl(\frac{R}{r}\biggr)^{m-1} R^{m-\pdmn/2} \doublebar{\nabla^m u}_{L^2(\R^\dmn)},
\quad \text{where }R=\abs{x-z_0}+r.\end{equation*}
\end{lem}
\begin{proof}
By Lemma~\ref{lem:Morrey}, 
\begin{align*}
\abs{u(x)}
&\leq 
	C\biggl(\sum_{k=0}^m R^{2k}\fint_{B(z_0,2R)} \abs{\nabla^k u}^2\biggr)^{1/2}
.\end{align*}
Let $ P(x)$ be the polynomial of degree at most $m-1$ such that 
\begin{equation*}\fint_{B(z_0,2R)} \partial^\gamma  P(x)\,dx = \fint_{B(z_0,2R)} \partial^\gamma u(x)\,dx\end{equation*} 
for all $\abs{\gamma}\leq m-1$.
Then
\begin{align*}
\abs{u(x)}
&\leq 
	C\biggl(\sum_{k=0}^m R^{2k}\fint_{B(z_0,2R)} \abs{\nabla^k u-\nabla^k P}^2
	+\sum_{k=0}^m R^{2k}\fint_{B(z_0,2R)} \abs{\nabla^k P}^2\biggr)^{1/2}
.\end{align*}
If $k \leq m-1$, then by the Poincar\'e inequality
\begin{equation*}R^{2k}\fint_{B(z_0,2R)} \abs{\nabla^k u-\nabla^k P}^2
\leq
R^{2 m}\fint_{B(z_0,2R)} \abs{\nabla^{ m} u}^2.\end{equation*}
Therefore,
\begin{align*}
\abs{u(x)}
&\leq 
	CR^{m-\pdmn/2}\biggl(\int_{B(z_0,2R)} \abs{\nabla^m u}^2\biggr)^{1/2}
	+
	C\biggl(\sum_{k=0}^m R^{2k}\fint_{B(z_0,2R)} \abs{\nabla^k P}^2\biggr)^{1/2}
.\end{align*}
By Lemma~\ref{lem:Morrey} and the above bounds on $\nabla^k u-\nabla^k P$, if $1\leq i\leq q$ then
\begin{align*}
\abs{ P(z_0+r\,h_i)}
&=\abs{ P(z_0+r\,h_i)-u(z_0+r\,h_i)}
\leq 
	CR^{m-\pdmn/2}\doublebar{\nabla^m u}_{L^2(B(z_0,2R))}
.\end{align*}
Let $ P(x)= Q((x-z_0)/r)$, so that $ Q(h_i)= P(z_0+r\,h_i)$. By construction of $Q$ and $h_i$, we have that
\begin{equation*}
Q(x)=\sum_{\abs{\gamma}\leq m-1} q_\gamma \,x^\gamma
\text{ for some $q_\gamma$ with }
\abs{q_\gamma} \leq CR^{m-\pdmn/2}\doublebar{\nabla^m u}_{L^2(B(z_0,2R))}.\end{equation*}
Then
\begin{equation*}\partial^\delta P(x)=\sum_{\gamma\geq \delta} r^{-\abs{\gamma}}q_\gamma \,(x-z_0)^{\gamma-\delta}
\end{equation*}
and so if $x\in B(z_0,2R)$, then
\begin{equation*}\abs{\nabla^k P}
\leq 
C \biggl(\frac{R}{r}\biggr)^{m-1} R^{-k}\sup_\gamma \abs{q_\gamma}
\leq
C \biggl(\frac{R}{r}\biggr)^{m-1} R^{m-k-\pdmn/2}
\doublebar{\nabla^m u}_{L^2(B(z_0,2R))}
.
\end{equation*}
Combining these estimates, we have that
\begin{align*}
\abs{u(x)}
&\leq 
	C
	\biggl(\frac{R}{r}\biggr)^{m-1}	R^{m-\pdmn/2}\doublebar{\nabla^m u}_{L^2(B(z_0,2R))}
\end{align*}
as desired.\end{proof}

We apply the lemma to the function $u=\vec\Pi^L_{z_0,r}\mydot F$.
Recall that $\nabla^m \vec\Pi^L$ is bounded on $L^2(\R^\dmn)$, and so
\begin{equation*}\abs{\vec S_x\mydot F}
\leq CR^{m-\pdmn/2} \biggl(\frac{R}{r}\biggr)^{m-1} \doublebar{\mydot F}_{L^2({\R^\dmn})},
\quad R=\abs{x-z_0}+r.
\end{equation*}

By the Riesz representation theorem, there is some array $\mat E^L$ such that
\begin{equation*}(\vec\Pi^L_{z_0,r}\mydot F)_j(x)
=(\vec S_x\mydot F)_j
=\sum_{k=1}^N\sum_{\abs{\beta}=m} \int_{\R^\dmn} E^L_{j,k,\beta,z_0,r}(x,y)\,F_{k,\beta}(y)\,dy.\end{equation*}
Furthermore, $\mat E^L$ satisfies the bound
\begin{equation}
\label{eqn:E:high:bound}
\doublebar{E^L_{j,k,\beta,z_0,r}(x,\,\cdot\,)}_{L^2({\R^\dmn})} \leq  C(p) R^{m-\pdmn/2} \biggl(\frac{R}{r}\biggr)^{m-1} 
,\quad R=r+\abs{x-z_0}
.\end{equation}
As in the proof of Lemma~\ref{lem:newton:adjoint}, we may use Lemma~\ref{lem:curlfree} to see that there is some function $E^L_{j,k,z_0,r}$ such that $E^L_{j,k,\beta,z_0,r}(x,y)=\partial_y^\beta E^L_{j,k,z_0,r}(x,y)$.
Again $E^L_{j,k,z_0,r}(x,y)$ is not unique; we may fix a normalization by requiring that
\begin{equation*}\mat E^L_{z_0,r}(x,z_0+r\,h_i)=0\quad\text{for all }x\in\R^\dmn\text{ and all } 1\leq i\leq q.
\end{equation*}
Notice that by construction of $\mat E^L_{z_0,r}$, 
\begin{equation*}
\partial_y^\beta\mat E^L_{z_0,r}(z_0+r\,h_i,y)=0\quad\text{for all } 1\leq i\leq p
\end{equation*}
as an $L^2(\R^\dmn)$-function; thus $P(y)=\mat E^L_{z_0,r}(z_0+r\,h_i,y)$ is a polynomial in $y$ of order $m-1$, and because it is equal to zero at the points $y=z_0+r\,h_i$ we have that 
\begin{equation*}
\mat E^L_{z_0,r}(z_0+r\,h_i,y)=0\quad\text{for all } y\in\R^\dmn \text{ and all } 1\leq i\leq p.
\end{equation*}

We also observe that by Lemma~\ref{lem:high:Morrey:bound} and the bound \eqref{eqn:E:high:bound}, we have that
\begin{equation}
\label{eqn:E:z0:r:bound}
\abs{\mat E^L_{z_0,r}(x,y)}
\leq
	C r^{2m-\pdmn}
	\biggl(1+\frac{\abs{y-z_0}}{r}\biggr)^{2m-\pdmn/2-1} \biggl(1+\frac{\abs{x-z_0}}{r}\biggr)^{2m-\pdmn/2-1}
.\end{equation}

We have established the existence of $\mat E^L$ and the relation \eqref{eqn:fundamental}.
To complete the proof of Theorem~\ref{thm:fundamental:high}, we must show that the derivatives $\partial_x^\zeta\partial_y^\xi\mat E^L_{j,k,z_0,r}(x,y)$ exist in the weak sense and satisfy the bounds \eqref{eqn:fundamental:bound} and~\eqref{eqn:fundamental:bound:m-1}, and must establish the symmetry property~\eqref{eqn:fundamental:symmetric}.


Let $\eta$ be a smooth cutoff function, that is, $\int_{\R^\dmn}\eta=1$, $\eta\geq 0$ and $\eta\equiv 0$ outside of the unit ball $B(0,1)$. Let $\eta_\varepsilon(x)=\varepsilon^{-\pdmn}\eta(x/\varepsilon)$. We will let $*_x$ denote convolution in the $x$ variable, that is,
\begin{equation*}\eta_\varepsilon*_x E^L_{j,k,z_0,r}(x,y)=\int_{\R^\dmn} \eta_{\varepsilon}(\tilde x) \, E^L_{j,k,z_0,r}(x-\tilde x,y)\, d\tilde x.\end{equation*}
For the sake of symmetry we will consider the function $\eta_\delta*_x E^L_{j,k,z_0,r}(x,y)*_y \eta_\varepsilon$ for some $\varepsilon$, $\delta>0$.

For any multiindices $\zeta$ and~$\xi$, let
\begin{equation*}E^L_{j,k,\zeta,\xi,\delta,\varepsilon}(x,y)=\partial_x^\zeta \partial_y^\xi (\eta_\delta*_x E^L_{j,k,z_0,r}(x,y)*_y \eta_\varepsilon).\end{equation*}
We will then construct $\partial_x^\zeta \partial_y^\xi E^L_{j,k}(x,y)$ as the weak limit of $E^L_{j,k,\zeta,\xi,\delta,\varepsilon}(x,y)$ as $\varepsilon\to 0$, $\delta\to 0$.

We begin with the derivatives of highest order. Let $\abs{\alpha}=\abs{\beta}=m$.
Observe that 
\begin{equation*}
E^L_{j,k,\alpha,\beta,\delta,\varepsilon}(x,y)
=
(\partial^\alpha\eta_\delta)*_x E^L_{j,k,\beta,z_0,r}(x,y)*_y \eta_\varepsilon.
\end{equation*}
Now, we have that
\begin{align*}
\int_{\R^d} E^L_{j,k,\alpha,\beta,\delta,\varepsilon}(x,y)\, F(y)\,dy
&= 
	(\partial^\alpha\eta_\delta)*_x\int_{\R^d} E^L_{j,k,\beta,z_0,r}(x,y) \, (\eta_\varepsilon*F)(y)\,dy
\\&=
	\eta_\delta * \partial^\alpha \Pi^L_j (\eta_\varepsilon*F\,\mydot e_{k,\beta})(x)
.\end{align*}
The operator $F\mapsto \eta_\delta * \partial^\alpha \Pi^L_j (\eta_\varepsilon*F\,\mydot e_{k,\beta})(x)$ is bounded $L^2(\R^\dmn)\mapsto\C$, albeit with a bound depending on $\delta$. Thus by the Riesz representation theorem, $K(y)=E^L_{j,k,\alpha,\beta,\delta,\varepsilon}(x,y)$ is the kernel of this operator, and so does not depend on $z_0$ and~$r$. Furthermore, by Lemma~\ref{lem:newton:adjoint},
\begin{equation*}E^L_{j,k,\alpha,\beta,\delta,\varepsilon}(x,y)
=\overline{E^{L^*}_{k,j,\beta,\alpha,\varepsilon,\delta}(y,x)}.
\end{equation*}
In order to establish the bounds \eqref{eqn:fundamental:bound} and \eqref{eqn:fundamental:bound:m-1}, we would like to use the Caccioppoli inequality in both $x$ and~$y$; it will be helpful to have a similar symmetry relation for $\mat E^L_{z_0,r}(x,y)$ as well as its highest derivatives.


\begin{lem}
\label{lem:symmetric:lowest}
We have that
$E^L_{j,k,z_0,r}(x,y)=\overline{E^{L^*}_{k,j,z_0,r}(y,x)}.$
\end{lem}
\begin{proof}
Because $E^L_{j,k,\alpha,\beta,\delta,\varepsilon}(x,y)
=\overline{E^{L^*}_{k,j,\beta,\alpha,\varepsilon,\delta}(y,x)}$, we have that
\begin{equation*}\nabla_x^m E^L_{j,k,0,\beta,\delta,\varepsilon}(x,y)
=\nabla_x^m \overline{E^{L^*}_{k,j,\beta,0,\varepsilon,\delta}(y,x)}.\end{equation*}
Thus $E^L_{j,k,0,\beta,\delta,\varepsilon}(x,y)$ and $\overline{E^{L^*}_{k,j,\beta,0,\varepsilon,\delta}(y,x)}$ differ by a polynomial in~$x$ of order $m-1$. But observe that
\begin{equation*}E^L_{j,k,0,\beta,\delta,\varepsilon}(z_0+r\,h_i,y)
=0 = \overline{E^{L^*}_{k,j,\beta,0,\varepsilon,\delta}(y,z_0+r\,h_i)}\end{equation*}
for all $1\leq i\leq q$; by construction of the points~$h_i$, this implies that 
\begin{equation*}E^L_{j,k,0,\beta,\delta,\varepsilon}(x,y)
=\overline{E^{L^*}_{k,j,\beta,0,\varepsilon,\delta}(y,x)}.\end{equation*}
By a similar argument, 
\begin{equation*}E^L_{j,k,0,0,\delta,\varepsilon}(x,y)
=\overline{E^{L^*}_{k,j,0,0,\varepsilon,\delta}(y,x)}.\end{equation*}
By Morrey's inequality $\mat E^L$ is continuous. Taking the limits as $\varepsilon\to 0$ and $\delta \to 0$ completes the proof.
\end{proof}

We now wish to establish an $L^2$ bound on $E^L_{j,k,\zeta,\xi,\delta,\varepsilon}$, independent of $\delta$ and~$\varepsilon$; this will allow us to prove the bounds \eqref{eqn:fundamental:bound} and \eqref{eqn:fundamental:bound:m-1}, and also to construct the derivatives by taking the limits as $\delta$, $\varepsilon\to 0$.
We will use the Caccioppoli inequality.

The first step is to show that $\mat E^L_{z_0,r}$ is a solution in some sense. Recall that if $\vec\varphi\in \dot W^2_m(\R^\dmn)$, then by formula~\eqref{eqn:newton:identity} $\varphi_j(x)=\Pi^L_j (\mat A\nabla^m\vec \varphi)(x)$, and so by our construction of $\mat E^L$,
\begin{equation*}\varphi_j(x)=
\sum_{k=1}^N\sum_{\abs{\beta}=m} \int_{\R^\dmn} \partial_y^\beta E^L_{j,k,z_0,r}(x,y)\,
\sum_{\ell=1}^N \sum_{\abs{\gamma}=m} A^{k\ell}_{\beta\gamma} \partial^\gamma \varphi_\ell(y)\,dy
\end{equation*}
as $\dot W^2_m$ functions; if $\vec\varphi(z_0+r\,h_i)=0$ for all $1\leq i\leq q$, then this equation is true pointwise for all~$x$. Thus, we have that for any $x$, $j$, $z_0$, $r$, the function $\vec v(y)$ given by 
$v_k(y)=E^L_{j,k,z_0,r}(x,y)$
is a solution to $L^*\vec v =0$ in $\R^d\setminus\{x\}\setminus \overline{B(z_0,r)}$.

Fix some $x_0$, $y_0$. We wish to bound $E^L_{j,k,\zeta,\xi,\delta,\varepsilon}$. Choose $z_0$ and $r$ so that $\abs{x_0-y_0}=\abs{x_0-z_0}=\abs{y_0-z_0}=8r$.

For any $x\in B(x_0,r)$, we have by Corollary~\ref{cor:Caccioppoli:interior}, if $\varepsilon$ is small compared to~$r$ then
\begin{align*}
\int_{B(y_0,r)} \abs{\mat E^L_{\zeta,\xi,\delta,\varepsilon}(x,y)}^2\,dy
&=
	\int_{B(y_0,r)} \abs{\eta_\varepsilon*_y(\partial_y^\xi (\partial^\zeta\eta_\varepsilon*_x\mat E^L_{z_0,r}(x,y))}^2\,dy
\\&\leq
	\int_{B(y_0,2r)} \abs{(\partial_y^\xi (\partial^\zeta\eta_\varepsilon*_x\mat E^L_{z_0,r}(x,y))}^2\,dy
\\&\leq
	\frac{C}{r^{2\abs{\xi}}}
	\int_{B(y_0,4r)} \abs{(\partial^\zeta\eta_\varepsilon*_x\mat E^L_{z_0,r}(x,y))}^2\,dy.
\end{align*}
Again by Corollary~\ref{cor:Caccioppoli:interior} and by the bound~\eqref{eqn:E:z0:r:bound},
\begin{align*}\int_{B(x_0,r)} \abs{(\partial^\zeta\eta_\varepsilon*_x\mat E^L_{z_0,r}(x,y))}^2\,dx
&=
	\int_{B(x_0,r)} \abs{(\eta_\varepsilon*_x \partial^\zeta_x E^{L^*}_{z_0,r}(y,x))}^2\,dx
\\&\leq
	\int_{B(x_0,2r)} \abs{\partial^\zeta_x \mat E^{L^*}_{z_0,r}(y,x)}^2\,dx
	\leq C r^{4m-\pdmn-2\abs{\zeta}}
.\end{align*}
Thus
\begin{multline*}
\int_{B(x_0,r)}\int_{B(y_0,r)} \abs{\mat E^L_{\zeta,\xi,\delta,\varepsilon}(x,y)}^2\,dy\,dx
\\\leq
	\frac{C}{r^{2\abs{\xi}}}
	\int_{B(y_0,4r)} \int_{B(x_0,r)}
	\abs{(\partial^\zeta\eta_\varepsilon*_x\mat E^L_{z_0,r}(x,y))}^2\,dx\,dy
	\leq C r^{4m-2\abs{\zeta}-2\abs{\xi}}.
\end{multline*}
So $\mat E^L_{\zeta,\xi,\delta,\varepsilon}$ is in $L^2(B(x_0,r)\times B(y_0,r))$, uniformly in $\delta$, $\varepsilon$; thus there is a weakly convergent subsequence as $\delta$, $\varepsilon\to 0$. Observe that the weak limit must be the partial derivative $\partial_x^\zeta \partial_y^\xi \mat E^L_{z_0,r}(x,y)$, as desired.

\subsection{Natural normalization conditions for the fundamental solution}
\label{sec:fundamental:further}


Recall that our normalization of $\mat E^L$, in the construction given in Section~\ref{sec:fundamental:high}, is highly artificial and depends on our choice of the normalization points $z_0+r\,h_i$.
In this section we will construct a somewhat more natural normalization of at least the higher derivatives of $\mat E^L$.

Our normalization will, loosely speaking, be a requirement that the higher-order derivatives of $\mat E^L$ decay at infinity.
Thus, we begin with a decay result.
\begin{lem}\label{lem:fundamental:far:m}
Let $A(x_0,R)$ denote the annulus $B(x_0,2R)\setminus B(x_0,R)$.
Let $p^+=\min(p^+_L,p^+_{L^*})$, where $p^+_L$ is as in Theorem~\ref{thm:Meyers}. If $0<\varepsilon<\pdmn(1-2/p^+)$, then there is some constant $C=C(\varepsilon)$ such that if $x_0\in\R^\dmn$ and $R>4r>0$, then
\begin{align*}
\int_{A(x_0,R)}\int_{B(x_0,r)} \abs{\nabla^m_x \nabla^m_y \mat E^L(x,y)}^2\,dy \,dx
&\leq C(\varepsilon) \biggl(\frac{r}{R}\biggr)^\varepsilon
.\end{align*}
\end{lem}

\begin{proof}
Let $\eta_\delta$ be a smooth approximate identity, as in Section~\ref{sec:fundamental:high}; we will establish a bound on $\eta_\delta *_x \nabla_x^m\nabla_y^m \mat E^L(x,y)$, uniform in $\delta$, and then let $\delta\to 0$.

Fix some $\delta>0$, $x\in\R^\dmn$, and some $j$ and $\alpha$ with $1\leq j\leq N$ and $\abs{\alpha}=m$. Let
\begin{equation*}v_k^\delta(y)=\eta_\delta *_x \partial_x^\alpha E^L_{j,k}(x,y).\end{equation*}
As in Section~\ref{sec:fundamental:high}, we begin by showing  that $\vec v^\delta$ is a solution to an elliptic equation. By the bound \eqref{eqn:fundamental:bound}, we have that 
$\vec v^\delta\in \dot W^2_{m,loc}$.
Suppose that $\vec\varphi$ is smooth and compactly supported. If $\dist(x,\supp \vec\varphi)>\delta$, then by formula~\eqref{eqn:newton:identity} and formula~\eqref{eqn:fundamental}, 
\begin{equation*}0=\eta_\delta *\partial^\alpha\varphi_j(x)=
\sum_{k=1}^N \sum_{\ell=1}^N \sum_{\abs{\beta}=\abs{\gamma}=m} \int_{\R^\dmn} 
\eta_\delta *_x\partial_x^\alpha \partial_y^\beta E^L_{j,k}(x,y)\,A^{\beta\gamma}_{k\ell}(y) \partial^\gamma \varphi_\ell(y)\,dy
.\end{equation*}
So $L^*\vec v^\delta=0$ in $\R^\dmn\setminus B(x,\delta)$, and so Theorem~\ref{thm:Meyers} applies.

Let $p$ be such that $\varepsilon=\pdmn(1-2/p)$; notice that $2<p<p^+$. 
By H\"older's inequality, we have that
\begin{equation*}
\int_{B(x_0,r)}\abs{\nabla_y^m \vec v^\delta(y)}^2\,dy
\leq
	Cr^\varepsilon
	\biggl(\int_{B(x_0,r)}\abs{\nabla_y^m \vec v^\delta(y)}^p\,dy\biggr)^{2/p}
.\end{equation*}
Because $R>4r$, we may replace the second integral by an integral over the ball~$B(x_0,R/4)$. We then apply Theorem~\ref{thm:Meyers}.
This yields the bound
\begin{equation*}
\int_{B(x_0,r)}\abs{\nabla_y^m \vec v^\delta(y)}^2\,dy
\leq
	C
	\frac{r^\varepsilon}{R^\varepsilon}
	\int_{B(x_0,R/2)}\abs{\nabla_y^m \vec v^\delta(y)}^{2}\,dy
\end{equation*}
uniformly in~$\delta$. Taking the limit as $\delta\to 0$ and applying the bound~\eqref{eqn:fundamental:bound}, we see that
\begin{multline*}
\int_{A(x_0,R)} \int_{B(x_0,r)}\abs{\nabla_x^m \nabla_y^m \mat E^L(x,y)}^2\,dy\,dx
\\\begin{aligned}
&\leq	
	\frac{Cr^\varepsilon}{R^\varepsilon}
	\int_{A(x_0,R)} 
	\int_{B(x_0,R/2)}\abs{\nabla_x^m \nabla_y^m \mat E^L(x,y)}^{2}\,dy\,dx
\leq	
	\frac{Cr^\varepsilon}{R^\varepsilon}
\end{aligned}\end{multline*}
as desired.
\end{proof}
Because $L^*$ is also elliptic, a similar bound is valid for $\mat E^{L^*}$.
Notice that by formula~\eqref{eqn:fundamental:symmetric}, we have that $\nabla_x^m\nabla_y^m E^L_{j,k}(x,y)=\overline{\nabla_x^m\nabla_y^m E^{L^*}_{k,j}(y,x)}$. Thus, a similar bound on $\mat E^L$ is valid with the roles of $x$ and~$y$ reversed.

Next, we use this bound to produce natural normalizations of certain higher-order derivatives.

\begin{lem}\label{lem:fundamental:far:low:m}
Suppose that $E$ is a function such that, for some $v\geq 0$, $c>0$, $\varepsilon>0$ and $t<\dmn+\varepsilon$, the decay estimate
\begin{align*}
\int_{y\in B(x_0,r)} \int_{x\in A(x_0,R)}\abs{\nabla^m_x \nabla^{v}_y  E(x,y)}^2 \,dx \,dy
&\leq c R^{t}\biggl(\frac{r}{R}\biggr)^\varepsilon
\end{align*}
is true for all $x_0\in\R^\dmn$ and all $R>4r>0$.

Then there is an array of functions $p_{\gamma}$ such that, if 
\begin{equation*} \widetilde E(x,y)=  E(x,y)+\sum_{m+t/2-\pdmn/2-\varepsilon/2<\abs{\gamma}\leq m-1} p_{\gamma}(y)\,x^\gamma
\end{equation*}
then there is a constant $C=C(\varepsilon)$ depending only on $\varepsilon$ such that, for all integers $q$ with 
$0\leq q\leq m$ and $q< \pdmn/2+\varepsilon/2-t/2$, we have that 
\begin{equation}
\label{eqn:fundamental:far:low:m}
\int_{y\in B(x_0,r)}\int_{x\in A(x_0,R)} \abs{\nabla^{m-q}_x \nabla^{v}_y \widetilde E(x,y)}^2 \,dx\,dy\leq C(\varepsilon) \,c\,R^{t+2q}\biggl(\frac{r}{R}\biggr)^{\varepsilon}\end{equation}
for all $x_0\in\R^\dmn$ and all $R>4r>0$. 

Furthermore, $p_{\gamma}(y)$ is unique up to adding polynomials of order $v-1$.
\end{lem}

By Lemma~\ref{lem:fundamental:far:m}, if $E=E^L_{j,k}$ is a component of the fundamental solution for some elliptic operator~$L$, then $E$ satisfies the conditions of Lemma~\ref{lem:fundamental:far:m} for $v=m$ and $t=0$; we will shortly need the lemma for $v<m$ as well.

\begin{proof}[Proof of Lemma~\ref{lem:fundamental:far:low:m}]
We begin with uniqueness. Suppose that there were two such arrays $p$ and $\tilde p$. Let $P_{\gamma}(y)=p_\gamma(y)-\tilde p_\gamma(y)$. If  $m+t/2-\pdmn/2-\varepsilon/2<\abs{\gamma}$ and $\abs{\gamma}\leq m-1$, then the difference $P_{\gamma}(y)\,x^\gamma$ must satisfy the bound~\eqref{eqn:fundamental:far:low:m} for  $q=m-\abs{\gamma}$. Thus, for any $x_0\in\R^\dmn$ and any $R>4r>0$, we have that
\begin{equation*}
\int_{y\in B(x_0,r)}\int_{x\in A(x_0,R)} \abs{\nabla^{\abs{\gamma}}_x \nabla^v_y (P_{\gamma}(y)\,x^\gamma)}^2 \,dx\,dy
\leq C(r,\varepsilon)\,c \,R^{t+2m-2\abs{\gamma}-\varepsilon}.
\end{equation*}
But
\begin{equation*}
\int_{y\in B(x_0,r)}\int_{x\in A(x_0,R)} \abs{\nabla^{\abs{\gamma}}_x \nabla^v_y (P_{\gamma}(y)\,x^\gamma)}^2 \,dx\,dy
=
CR^\dmn\int_{B(x_0,r)}\abs{\nabla^v_y P_{\gamma}(y)}^2\,dy
.\end{equation*}
Because $m+t/2-\pdmn/2-\varepsilon/2<\abs{\gamma}$, we have that ${2m+t-2\abs{\gamma}-\varepsilon}<\dmn$ and so $R^\dmn$ grows faster than $R^{2m+t-2\abs{\gamma}-\varepsilon}$. Thus, the only way that both conditions can hold is if 
$\nabla^{v}_y P_{\gamma}(y)=0$ almost everywhere in $B(x_0,r)$. Since $x_0$ and $r$ were arbitrary this means that $P_{\gamma}$ is a polynomial of order~$v-1$, as desired.

We now construct an appropriate array of functions $p_{\gamma}(y)$.
We work by induction; notice that by assumption, the bound \eqref{eqn:fundamental:far:low:m} is valid in the case $q=0$. 

Choose some $q>0$ satisfying the conditions of the lemma, and suppose that we have renormalized $E$ so that the bound \eqref{eqn:fundamental:far:low:m} is valid if we replace $q$ by $q-1$.
Choose some multiindices $\gamma$ and $\zeta$ with $\abs{\gamma}=m-q$ and $\abs{\zeta}=v$.

Let $A_i=B(x_0,2^i)\setminus B(x_0,2^{i-1})$, and define
\begin{equation*}
E_i(y)=\fint_{A_i}\partial_x^\gamma \partial_y^\zeta E(x,y)\,dx
.\end{equation*}
For any constant $c_i$ we have the bound
\begin{align*}\abs{ E_i(y)- E_{i+1}(y)}
&=\abs[bigg]{
	\fint_{A_i}\partial_x^\gamma \partial_y^\zeta   E(x,y)\,dx 	-\fint_{A_{i+1}}\partial_x^\gamma \partial_y^\zeta E(x,y)\,dx }
\\&\leq
\abs[bigg]{\fint_{A_i}\partial_x^\gamma \partial_y^\zeta E(x,y)\,dx -c_i}
+
\abs[bigg]{\fint_{A_{i+1}}\partial_x^\gamma \partial_y^\zeta E(x,y)\,dx-c_i}
\\&\leq
C\fint_{A_i\cup A_{i+1}}\abs{\partial_x^\gamma \partial_y^\zeta E(x,y)-c_i}\,dx
.\end{align*}
Choosing $c_i$ appropriately, by Poincar\'e's inequality,
\begin{equation*}
\abs{ E_i(y)- E_{i+1}(y)}
\leq
C 2^{-i\pdmnMinusOne}
\int_{A_i\cup A_{i+1}}\abs{\nabla_x^{m-q+1} \nabla_y^{v}E(y,x)}\,dx
.\end{equation*}
Thus by H\"older's inequality
\begin{multline*}
\int_{B(x_0,r)} \abs{ E_i(y)- E_{i+1}(y)}^2\,dy
\\\begin{aligned}
&\leq 
\frac{C}{2^{2i\pdmnMinusOne}}\int_{B(x_0,r)}
\biggl(\int_{A_i\cup A_{i+1}}\abs{ \nabla_x^{m-q+1}\nabla_y^{v} E(y,x)}\,dx\biggr)^2\,dy
\\&\leq 
\frac{C}{2^{i(\dmn-2)}}\int_{B(x_0,r)}
\int_{A_i\cup A_{i+1}}\abs{\nabla_x^{m-q+1}\nabla_y^{v} E(y,x)}^2\,dx\,dy
.\end{aligned}\end{multline*}
Recall that we assumed that we had the desired decay estimates for $q-1$; this implies that
\begin{equation*}\int_{B(x_0,r)} \abs{ E_i(y)- E_i(y)}^2\,dy
\leq 
C\,c\, 2^{i(t-\pdmn+2q-\varepsilon)}
r^{\varepsilon} 
.\end{equation*}
Thus, by our conditions on $q$, $E_\infty(y)=\lim_{i\to\infty} E_i(y)$ exists as an $L^2(B(x_0,r))$-function. As usual we may use Lemma~\ref{lem:curlfree} to see that there is some $p_{\gamma}(y)$ such that $E_\infty(y)=\gamma!\partial^\zeta p_\gamma(y)$. Let $\widetilde E(x,y)=E(x,y)-p_\gamma(y)\,x^\gamma$.

We construct an $\widetilde  E_i$ from $\widetilde E$, similar to our construction of~$E_i$; then $\widetilde E_i$ satisfies the same bounds as above and converges to zero as $i\to\infty$.
Because geometric series converge, we have that
\begin{equation*}\int_{B(x_0,r)} \abs{\widetilde E_i(y)}^2\,dy \leq C(\varepsilon)\,c\, 2^{i(t-\pdmn+2q-\varepsilon)}
r^{\varepsilon}.\end{equation*}

By the Poincar\'e inequality
\begin{multline*}
\int_{B(x_0,r)}\int_{A_i}  \abs{\partial_x^\gamma \partial_y^\zeta \widetilde E^L_{j,k}(x,y)}^2\,dy\,dx
\\\begin{aligned}
&\leq 
	C
	\int_{B(x_0,r)}\int_{A_i}  \abs{\partial_x^\gamma \partial_y^\zeta \widetilde E^L_{j,k}(x,y)-\widetilde E_i(x)}^2\,dx\,dy
	+
	C(\varepsilon)\,c\, 2^{i(t+2q-\varepsilon)}r^{\varepsilon}
\\&\leq 
	C 2^{2i}
	\int_{B(x_0,r)}\int_{A_i}  \abs{\nabla_x^{m-q+1} \partial_y^\zeta \widetilde E^L_{j,k}(x,y)}^2\,dx\,dy
	+
	C(\varepsilon)\,c\, 2^{i(t+2q-\varepsilon)}r^{\varepsilon}
\\&\leq 
	C(\varepsilon)\,c\, 2^{i(t+2q-\varepsilon)}r^{\varepsilon}
\end{aligned}\end{multline*}
as desired. Repeating this construction for all $\gamma$ with $\abs{\gamma}=m-q$, we complete the proof.
\end{proof}

By Lemmas~\ref{lem:fundamental:far:m} and~\ref{lem:fundamental:far:low:m},
there is a unique appropriately normalized representative of $\nabla_x^{m-q}\nabla_y^m\mat E^L(x,y)$. 
Recall that by formula~\eqref{eqn:fundamental:symmetric}, we have that $\mat E^L(x,y)$ satisfies the conclusion of Lemma~\ref{lem:fundamental:far:m} with the roles of $x$ and~$y$ reversed. We may thus find a unique additive normalization of $\nabla_x^{m}\nabla_y^{m-q}\mat E^L(x,y)$. Also notice that by formula~\eqref{eqn:fundamental:symmetric}, applying the same procedure to $\mat E^{L^*}$, we see that this normalization preserves the relations
\begin{gather*}
\nabla_x^{m-q}\nabla_y^m E^L_{j,k}(x,y) = \overline{\nabla_x^{m-q}\nabla_y^m E^{L^*}_{k,j}(y,x)}
,\\
\nabla_x^{m}\nabla_y^{m-q} E^L_{j,k}(x,y) = \overline{\nabla_x^{m}\nabla_y^{m-q} E^{L^*}_{k,j}(y,x)}
.\end{gather*}

We are now interested in the mixed derivatives, that is, in the case where we take fewer than $m$ derivatives in both $x$ and~$y$.

Observe first that if $q<\pdmn(1-1/p^+)$ and if $x_0\in\R^\dmn$, $y_0\in\R^\dmn$, then
\begin{equation*}\int_{B(y_0,R)}\int_{B(x_0,R)} \abs{\nabla^{m}_x \nabla^{m-q}_y \mat E^L(x,y)}^2\,dx\,dy \leq C R^{2q},
\quad R=\abs{x_0-y_0}/3
.\end{equation*}
As in the proof of 
Lemma~\ref{lem:fundamental:far:m}, we may use H\"older's inequality and Theorem~\ref{thm:Meyers} to see that
\begin{equation*}\int_{x\in A(x_0,R)} \int_{y\in B(x_0,r)}\abs{\nabla^{m-q}_y \nabla^{m}_x \mat E^L(x,y)}^2\,dy \,dx\leq C(\varepsilon) R^{2q}\biggl(\frac{r}{R}\biggr)^\varepsilon
\end{equation*}
for all $0<\varepsilon<\pdmn(1-2/p_q^+)$, where $p_q^+=p^+\pdmn/(\dmn-q\,p^+)$ in the case $q<\pdmn/p^+$ and $p_q^+=\infty$ if $\pdmn/p^+\leq q<\pdmn/p^-$. We may rewrite this requirement as $0<\varepsilon<\min(\dmn, \pdmn(1-2/p^+)+2q)$.

We may thus apply Lemma~\ref{lem:fundamental:far:low:m} with $v=m-q$ and $t=2q$.
Hence, if $q$ and $\varepsilon$ are as above, and if 
$ s<\pdmn/2+\varepsilon/2-q$,
then there is a unique additive normalization of $\nabla^{m-q}_y\nabla^{m-s}_x \mat E^L(x,y)$ such that
\begin{equation}
\label{eqn:fundamental:far:low}
\int_{B(x_0,r)}\int_{A(x_0,R)} \abs{\nabla^{m-s}_x \nabla^{m-q}_y \mat E^L(x,y)}^2\,dx\,dy \leq C(\varepsilon) R^{2q+2s}\biggl(\frac{r}{R}\biggr)^\varepsilon
.\end{equation}
We remark that we may find an appropriate $\varepsilon$ if and only if $q$ and~$s$ satisfy the conditions
$0\leq q\leq m$, $0\leq s\leq m$, $q<\pdmn/p^-$, $s<\pdmn/p^-$, and $q+s<\dmn$.

We will establish one more bound on the fundamental solution. Specifically, notice that $\nabla^m_x\nabla^m_y\mat E^L(x,y)$ is only locally integrable away from the diagonal $\{(x,y):x=y\}$. The lower-order derivatives, however, are locally integrable even near $x=y$.

\begin{lem}\label{lem:fundamental:near}
Let $q$ and $s$ be such that $0<q+s<\dmn$ and such that the bound \eqref{eqn:fundamental:far:low} is valid for all $x_0\in\R^\dmn$ and all $R>4r>0$.

Suppose that $p<\pdmn/(\dmn-(q+s))$ and that $p\leq 2$. We then have the local estimate
\begin{equation*}\int_{B(x_0,r)}\int_{B(x_0,r)} \abs{\nabla^{m-s}_x \nabla^{m-q}_y \mat E^L(x,y)}^p\,dx\,dy \leq C r^{2\pdmn-p(\dmn-s-q)}.\end{equation*}
\end{lem}

\begin{proof}
Let $Q_0$ be the cube of sidelength $\ell(Q_0)=2r$ with $B(x_0,r)\subset Q_0$, so that 
\begin{multline*}
\int_{B(x_0,r)}\int_{B(x_0,r)} \abs{\nabla^{m-s}_x \nabla^{m-q}_y \mat E^L(x,y)}^p\,dx\,dy
\\\begin{aligned}
&\leq
	\int_{Q_0}\int_{2Q_0} \abs{\nabla^{m-s}_x \nabla^{m-q}_y \mat E^L(x,y)}^p\,dx\,dy
.\end{aligned}\end{multline*}
We divide $Q_0$ as follows. Let $\mathcal{G}_j$ be a grid of dyadic subcubes of $Q_0$ of sidelength $2^{1-j}r$. Notice that $\mathcal{G}_0=\{Q_0\}$ and that ${\mathcal{G}_j}$ contains $2^{j\pdmn}$ cubes.

If $y\in B(x_0,r)$, let $Q_j(y)$ be the cube that satisfies $y\in Q_j(y)\in \mathcal{G}_j$. If $Q\in \mathcal{G}_{j+1}$, let $P(Q)$ be the unique cube with $Q\subset P(Q)\in \mathcal{G}_j$. If $Q$ is a cube, let $2Q$ be the concentric cube with side-length $\ell(2Q)=2\ell(Q)$. Then
\begin{multline*}
\int_{Q_0}\int_{2Q_0} \abs{\nabla^{m-s}_x \nabla^{m-q}_y \mat E^L(x,y)}^p\,dx\,dy
\\\begin{aligned}
&=
	\int_{Q_0}
	\sum_{j=0}^\infty
	\int_{2Q_j(y)\setminus 2Q_{j+1}(y)} \abs{\nabla^{m-s}_x \nabla^{m-q}_y \mat E^L(x,y)}^p\,dx\,dy
\\&=
	\sum_{j=0}^\infty
	\sum_{Q\in \mathcal{G}_{j+1}}
	\int_Q
	\int_{2P(Q)\setminus 2Q} \abs{\nabla^{m-s}_x \nabla^{m-q}_y \mat E^L(x,y)}^p\,dx\,dy
.\end{aligned}\end{multline*}
We apply H\"older's inequality to see that 
\begin{multline*}
\int_Q
	\int_{2P(Q)\setminus 2Q} \abs{\nabla^{m-s}_x \nabla^{m-q}_y \mat E^L(x,y)}^p\,dx\,dy
\\	\leq
	C\ell(Q)^{\dmn(2-p)}
	\biggl(\int_Q
	\int_{2P(Q)\setminus 2Q} \abs{\nabla^{m-s}_x \nabla^{m-q}_y \mat E^L(x,y)}^2\,dx\,dy\biggr)^{p/2}
\end{multline*}
and the bound \eqref{eqn:fundamental:far:low} to see that
\begin{equation*}
\int_Q
	\int_{2P(Q)\setminus 2Q} \abs{\nabla^{m-s}_x \nabla^{m-q}_y \mat E^L(x,y)}^p\,dx\,dy
	\leq
	C\ell(Q)^{\dmn(2-p)+(q+s)p}
.\end{equation*}
Combining these estimates and recalling that there are $2^{j\dmn}$ cubes $Q\in \mathcal{G}_j$, we see that
\begin{equation*}
\int_{Q_0}\int_{2Q_0} \abs{\nabla^{m-s}_x \nabla^{m-q}_y \mat E^L(x,y)}^p\,dx\,dy
\leq
	C r^{2\dmn-(\dmn-q-s)p}
	\sum_{j=0}^\infty
	2^{-j\dmn+j(\dmn-q-s)p}
.\end{equation*}
If $p<\pdmn/(\pdmn-(q+s))$, then the geometric series converges, as desired.
\end{proof}

We have renormalized the fundamental solution so that we may bound its lower-order derivatives. This renormalization will not affect the bound \eqref{eqn:fundamental:bound}, and because our renormalization is unique it maintains the symmetry condition \eqref{eqn:fundamental:symmetric}. 

Theorem~\ref{thm:fundamental:high} had one more conclusion, the formula~\eqref{eqn:fundamental}. This states that
\begin{equation*}
\Pi^L_{j}\mydot F(x) 
	= \sum_{k=1}^N \sum_{\abs{\beta}=m} \int_{\R^\dmn} 	\partial_y^\beta E^L_{j,k,z_0,r}(x,y)\,F_{k,\beta}(y)\,dy
	\quad\text{as $\dot W^2_m(\R^\dmn)$-functions}
	.
\end{equation*}
We would like to consider in what sense this equation is still true after renormalization. To address this, we will also need natural normalizations of the left-hand side $\vec\Pi^L\mydot F$ involving decay at infinity; this normalization is given by the following lemma.

\begin{lem}\textup{(The Gagliardo-Nirenberg-Sobolev inequality in $\R^\dmn$).}
\label{lem:GNS}
Let $ u$ lie in the space $\dot W^p_m(\R^\dmn)$ for some $1\leq p<\dmn$. 
Let $0<k<\pdmn/p$ be an integer, and let $p_k=p\,\pdmn/(\dmn-p\,k)$.

Then there is a unique additive normalization of $\nabla^{m-k} u$ in $L^{p_k}(\R^\dmn)$. 
\end{lem}

See, for example, Section~5.6.1 in \cite{Eva98}. We use this lemma to address the relation between the Newton potential and the renormalized fundamental solution.

\begin{lem}
\label{lem:renormalization}
Let $p^-<p<\min(\dmn,p^+)$, let $\gamma$ be a multiindex with $m-\pdmn/p<\abs{\gamma}\leq m-1$, and let $q>\pdmn/(m-\abs{\gamma})$.
Let $1\leq j\leq N$.

Suppose that we have normalized $\mat E^L$ as above. We normalize the lower-order derivatives of $\vec \Pi^L\mydot F$ as in Lemma~\ref{lem:GNS}. 
If $\mydot F$ lies in $L^p(\R^\dmn)$ and in $L^q_{loc}(\R^\dmn)$, then
\begin{equation}
\label{eqn:renormalization}
\partial^\gamma\Pi^L_{j} \mydot F(x)=\sum_{k=1}^N \sum_{\abs{\beta}=m} \int_{\R^\dmn} 	\partial_x^\gamma\partial_y^\beta E^L_{j,k}(x,y)\,F_{k,\beta}(y)\,dy\end{equation}
for almost every $x\in\R^\dmn$.
\end{lem}

\begin{proof}

Let us define 
\begin{align*}
\Pi^L_{j,\gamma} \mydot F(x)
&=
	\sum_{k=1}^N \sum_{\abs{\beta}=m} \int_{\R^\dmn} 	\partial_x^\gamma\partial_y^\beta E^L_{j,k}(x,y)\,F_{k,\beta}(y)\,dy
\end{align*}
where $\mat E^L$ is the fundamental solution normalized to obey the bound \eqref{eqn:fundamental:far:low}.
We begin by showing that $\Pi^L_{j,\gamma}$ is a bounded operator in some sense. Specifically, let $B(x_0,r)\subset\R^\dmn$ be a ball. We will show that $\Pi^L_{j,\gamma}$ is bounded $L^q(B(x_0,2r))\cap L^p(\R^\dmn)\mapsto L^1(B(x_0,r))$.

First, we see that
\begin{align*}
\int_{B(x_0,r)} \abs{\Pi^L_{j,\gamma}\mydot F(x)}\,dx
&\leq C\int_{B(x_0,r)} \int_{B(x_0,2r)}
 	\abs{\nabla_x^{\abs{\gamma}}\nabla_y^m \mat E^L(x,y)}\,\abs{\mydot F(y)}\,dy
	\,dx
\\&\qquad+ C\sum_{i=1}^\infty\int_{B(x_0,r)} \int_{A_i}
 	\abs{\nabla_x^{\abs{\gamma}}\nabla_y^m \mat E^L(x,y)}\,\abs{\mydot F(y)}\,dy
\,dx
\end{align*}
where $A_i=B(x_0,2^{i+1}r)\setminus B(x_0,2^ir)$.
If $1/q+1/q'=1$, then $q'<\pdmn/(\dmn-(m-\abs{\gamma}))$, and so by Lemma~\ref{lem:fundamental:near} and H\"older's inequality, the first integral is at most
\begin{align*}
C r^{\pdmn-\pdmn/q+m-\abs{\gamma}}
 \doublebar{\mydot F}_{L^q(B(x_0,2r))}
.\end{align*}
We control the second integral as follows. Fix some $i\geq 1$. Then by H\"older's inequality,
\begin{multline*}
\int_{B(x_0,r)}\int_{A_i}
 	\abs{\nabla_x^{\abs{\gamma}}\nabla_y^m \mat E^L(x,y)}\,\abs{\mydot F(y)}\,dy\,dx
\\\begin{aligned}
&\leq
	Cr^{\pdmn/p}\biggl(\int_{B(x_0,r)}\int_{A_i}
 	\abs{\nabla_x^{\abs{\gamma}}\nabla_y^m \mat E^L(x,y)}^{p'}\,dy\,dx\biggr)^{1/p'}
 	\doublebar{\mydot F}_{L^p(A(x_0,2^i r))}
.\end{aligned}\end{multline*}
Notice that $p'<p^+$. Arguing as in the proof of Lemma~\ref{lem:fundamental:far:m}, we use Theorem~\ref{thm:Meyers} to show that
\begin{multline*}
\biggl(\int_{B(x_0,r)}\int_{A_i}
 	\abs{\nabla_x^{\abs{\gamma}}\nabla_y^m \mat E^L(x,y)}^{p'}\,dy\,dx\biggr)^{1/p'}
\\\begin{aligned}
&\leq
C2^{i(\pdmn/p'-\pdmn/2)}r^{2\pdmn/p'-\pdmn}
\biggl(\int_{B(x_0,r)}
\int_{\widetilde A(x_0,2^ir)}
 	\abs{\nabla_x^{\abs{\gamma}}\nabla_y^m \mat E^L(x,y)}^{2}\,dy\,dx\biggr)^{1/2}
\end{aligned}
\end{multline*}
where $\widetilde A(x_0,2^ir)$ is the enlarged annulus $B(x_0,2^{i+2}r)\setminus B(x_0,(3/4) 2^i r)$. 

By the bound \eqref{eqn:fundamental:far:low},
\begin{multline*}
\biggl(\int_{B(x_0,r)}\int_{A_i}
 	\abs{\nabla_x^{\abs{\gamma}}\nabla_y^m \mat E^L(x,y)}^{p'}\,dy\,dx\biggr)^{1/p'}
\\\begin{aligned}
&\leq
C(\varepsilon)2^{i(\pdmn/p'-\pdmn/2+m-\abs{\gamma}-\varepsilon/2)}r^{2\pdmn/p'-\pdmn+m-\abs{\gamma}}
\end{aligned}
\end{multline*}
for all $0<\varepsilon < \min(\dmn,\pdmn(1-1/p^+) + 2m-2\abs{\gamma})$.
Let $\theta= \theta(\varepsilon) = -\pdmn/p'+\pdmn/2-m+\abs{\gamma}+\varepsilon/2$.
We remark that by our assumptions on $\gamma$ and~$p$, we may always find an $\varepsilon$ that satisfies the above conditions and such that $\theta>0$.

Thus,
\begin{align*}
\int_{B(x_0,r)} \abs{\Pi^L_{j,\gamma}\mydot F(x)}\,dx
&\leq Cr^{m-\abs{\gamma}+\pdmn/q'} \doublebar{\mydot F}_{L^q(B(x_0,2r))}
\\&\qquad
	+ C(\theta)r^{m-\abs{\gamma}+\pdmn/p'}
	\sum_{i=1}^\infty 
	2^{-i\theta} 
 	\doublebar{\mydot F}_{L^p(A(x_0,2^i r))}
\end{align*}
and by convergence of geometric series, we have that $\Pi^L_{j,\gamma}$ is bounded as an operator from $L^q(B(x_0,2r))\cap L^p(\R^\dmn)$ to $L^1(B(x_0,r))$, as desired.

We may now work in a dense subspace of $L^p(\R^\dmn)\cap L^q_{loc}(\R^\dmn)$; we will work with $\mydot F$ bounded and compactly supported.

In particular, suppose that $\mydot F$ is supported in some ball $B(y_0,r)$. Let $z_0$ be such that $\abs{y_0-z_0}=3r$, and consider the fundamental solution $\mat E^L_{z_0,r}$ of Theorem~\ref{thm:fundamental:high}; as in Section~\ref{sec:fundamental:high} we will let 
\begin{align*}
\Pi^L_{j,z_0,r} \mydot F(x)
&=
	\sum_{k=1}^N \sum_{\abs{\beta}=m} \int_{\R^\dmn} 	\partial_y^\beta E^L_{j,k}(x,y)\,F_{k,\beta}(y)\,dy.
\end{align*}



Begin with the case $\abs{\gamma}=m-1$. We will show that there is some constant $c$ such that $\Pi^L_{j,\gamma} \mydot F(x)=\partial^\gamma \Pi^L_{j,z_0,r} \mydot F(x)+c$ for almost every $x\in\R^\dmn$; it will then be straightforward to establish that $\Pi^L_{j,\gamma} \mydot F$ decays and so must equal the normalization of Lemma~\ref{lem:GNS}.

Observe that our renormalization of $\mat E^L$ preserves the relation \begin{equation*}\nabla_x^m\nabla_y^m\mat E^L(x,y)=\nabla_x^m\nabla_y^m\mat E^L_{z_0,r}(x,y).\end{equation*}
Thus by Lemma~\ref{lem:fundamental:far:low:m}, 
for every $\beta$ with $\abs{\beta}=m$ and every $j$, $k$, there is a unique function $p$ such that 
\begin{equation*}\partial_x^\gamma\partial_y^\beta E^L_{j,k}(x,y) = 
\partial_x^\gamma\partial_y^\beta E^L_{j,k,z_0,r}(x,y)
+p(y).\end{equation*} 
In particular, while $p$ may depend on $\gamma$, $\beta$, $j$, $k$, $z_0$ and~$r$, once these parameters are fixed, $p$ cannot depend on~$x$. It will be convenient to write $p=p_{k,\beta}$ and leave the remaining dependencies implied.

Let $x_0$ satisfy $\abs{x_0-y_0}=\abs{x_0-z_0}=3r$. Notice that
\begin{align*}
\int_{B(y_0,r)} \abs{p_{k,\beta}}^2
&=
	\int_{B(y_0,r)} \abs[bigg]{\fint_{B(x_0,r)} p_{k,\beta}(y)\,dx}^2 dy
\\&\leq 
	\int_{B(y_0,r)} \fint_{B(x_0,r)} \abs[big]{\partial_x^\gamma\partial_y^\beta E^L_{j,k}(x,y)-\partial_x^\gamma\partial_y^\beta E^L_{j,k,z_0,r}(x,y)}^2\, dx\,dy
\end{align*}
and so, using the bounds \eqref{eqn:fundamental:far:low} and \eqref{eqn:fundamental:bound:m-1}, we see that $p_{k,\beta}\in L^2(B(y_0,r))$ with $\doublebar{p_{k,\beta}}_{L^2(B(y_0,r))}\leq C r^{2-\pdmn}$.

Thus,
\begin{align*}
\Pi^L_{j,\gamma} \mydot F(x)
&=
	\sum_{k=1}^N \sum_{\abs{\beta}=m} \int_{\R^\dmn} 	\partial_x^\gamma\partial_y^\beta E^L_{j,k,z_0,r}(x,y)\,F_{k,\beta}(y)+p_{k,\beta}(y)\,F_{k,\beta}(y)\,dy
	.
\end{align*}
Notice that, by Lemma~\ref{lem:fundamental:near}, $\partial_x^\gamma\partial_y^\beta\mat E^L(x,y)\in L^1(U\times B(y_0,r))$ for any bounded set~$U$. If $U=B(x_0,r)$, then the inclusion $\partial_x^\gamma\partial_y^\beta\mat E^L_{z_0,r}(x,y)\in L^1(U\times B(y_0,r))$ follows from the bound \eqref{eqn:fundamental:bound:m-1}; because $\partial_x^\gamma\partial_y^\beta E^L_{j,k}(x,y) = 
\partial_x^\gamma\partial_y^\beta E^L_{j,k,z_0,r}(x,y)
+p_{k,\beta}(y)$, we may extend this second inclusion to all bounded sets~$U$. Thus
\begin{align*}
\Pi^L_{j,\gamma} \mydot F(x)
&=
	\sum_{k=1}^N \sum_{\abs{\beta}=m} \int_{\R^\dmn} 	\partial_x^\gamma\partial_y^\beta E^L_{j,k,z_0,r}(x,y)\,F_{k,\beta}(y)\,dy
	+
	\int_{\R^\dmn} 	p_{k,\beta}(y)\,F_{k,\beta}(y)\,dy
	.
\end{align*}
Observe that the second integral is convergent and also is independent of~$x$. Furthermore, we may apply Fatou's lemma to the first integral to see that 
\begin{align*}
\Pi^L_{j,\gamma} \mydot F(x)
&=
	c_1+\partial_x^\gamma\sum_{k=1}^N \sum_{\abs{\beta}=m} \int_{\R^\dmn} 	\partial_y^\beta E^L_{j,k,z_0,r}(x,y)\,F_{k,\beta}(y)\,dy
=
	c_1+\partial_x^\gamma\Pi^L_{j,z_0,r}\mydot F(x)
	.
\end{align*}
Because $\vec\Pi^L_{z_0,r}$ is an additive normalization of $\vec\Pi^L$, this implies that $\Pi^L_{j,\gamma} \mydot F(x)= c_2+\partial_x^\gamma\Pi^L_{j}\mydot F(x)$ where $\partial_x^\gamma\Pi^L_{j}\mydot F(x)$ is normalized as in Lemma~\ref{lem:GNS}.
We must now establish that $c_2=0$, that is, that $\Pi^L_{j,\gamma}\mydot F$ decays at infinity. But by the bound~\eqref{eqn:fundamental:far:low}, we have that
\begin{equation*}\lim_{R\to \infty} \fint_{A(y_0,R)} \abs{\Pi^L_{j,\gamma}\mydot F(x)}^2\,dx
=0\end{equation*}
and this can only be true for one additive normalization of $\partial^\gamma \Pi^L_{j} \mydot F$; it is this normalization that is chosen by Lemma~\ref{lem:GNS}, as desired.

We now consider $\abs{\gamma}<m-1$; we still work only with bounded, compactly supported functions~$\mydot F$. If $\abs{\gamma+\zeta}\leq m-1$, then by Fatou's lemma $\vec\Pi^L_{\zeta+\gamma}\mydot F= \partial^\zeta\vec\Pi^L_\gamma \mydot F$, and if $\abs{\gamma+\zeta}= m-1$ then by the above results $\vec\Pi^L_{\zeta+\gamma}\mydot F=\partial^{\zeta+\gamma} \vec\Pi^L\mydot F$.
Thus $\partial^\gamma\vec\Pi^L\mydot F=\vec\Pi^L_\gamma\mydot F$ up to adding polynomials. But again by the bound~\eqref{eqn:fundamental:far:low}, we have that
\begin{equation*}\lim_{R\to \infty} \fint_{A(y_0,R)} \abs{\Pi^L_{j,\gamma}\mydot F(x)}^2\,dx
=0\end{equation*}
whenever $m-\pdmn/p^-<\abs{\gamma}\leq m$; thus, $\partial^\gamma\vec\Pi^L\mydot F=\vec\Pi^L_\gamma\mydot F$, as desired.
\end{proof}

\begin{rmk}
\label{rmk:harmonic}
We have established decay results and the relation~\eqref{eqn:renormalization} only for the higher-order derivatives. We expect the lower-order derivatives to be problematic.
As an example, consider the case of the polyharmonic operator $L=(-\Delta)^m$; we may normalize the fundamental solution so that, for some constant $C_{m,\dmn}$,
\begin{equation*}E^{(-\Delta)^m}(x,y) =  
\begin{cases} 
	C_{m,\dmn} \abs{x-y}^{2m-\pdmn}, & \dmn \text{ odd or }\dmn >2m,\\
	C_{m,\dmn} \abs{x-y}^{2m-\pdmn}\log\abs{x-y}, & \dmn \text{ even and }\dmn \leq 2m.
\end{cases}\end{equation*}
Notice that $\partial_x^\zeta\partial_y^\xi E^{(-\Delta)^m}(x,y)$ decays at infinity only if $\abs{\zeta}+\abs{\xi}>2m-\pdmn$. Furthermore, if $\abs{\zeta}+\abs{\xi}=2m-\pdmn$, then no natural normalization condition applies; the fundamental solution given above must be normalized using deeper symmetry properties of the Laplacian and a choice of length scale for the logarithm.


In the case of more general operators, these symmetry properties are not available, and it is not apparent whether dimensionally-appropriate decay estimates are valid unless $\min(\abs{\zeta},\abs{\xi})>m-\pdmn+\pdmn/p^+$. Thus, in general, we do not have a unique normalization of the fundamental solution for operators of higher order.

We will see that we can construct a fundamental solution for operators of lower order and retain the above decay estimates, and in that case we will have a unique normalization of $\mat E^L$ provided $2m<\dmn$. (If $2m=\dmn$ then we will have unique normalizations of $\nabla_x\mat E^L(x,y)$ and $\nabla_y\mat E^L(x,y)$, and hence a normalization of $\mat E^L$ that is unique up to additive constants.)
\end{rmk}

\subsection{The fundamental solution for operators of lower order}
\label{sec:fundamental:low}

Consider the following theorem. In the case where $2m>\dmn$, validity of the following theorem was established in Sections~\ref{sec:fundamental:high} and~\ref{sec:fundamental:further}. In this section we will establish that Theorem~\ref{thm:fundamental:high:2} is still valid even if $2m\leq \dmn$.
\begin{thm}\label{thm:fundamental:high:2}
Let $L$ be an operator of order~$2m$ that satisfies the bounds \eqref{eqn:elliptic:bounded:2} and \eqref{eqn:elliptic}. 
Then there exists an array of functions $E^L_{j,k}(x,y)$ with the following properties.

Let $q$ and $s$ be two integers that satisfy $q+s<\dmn$ and the bounds $0\leq q\leq \min(m,\pdmn/2)$, $0\leq s\leq \min(m,\pdmn/2)$.

Then there is some $\varepsilon>0$ such that if $x_0\in\R^\dmn$, if $0<4r<R$, if $A(x_0,R)=B(x_0,2R)\setminus B(x_0,R)$, and if $q<\pdmn/2$ then 
\begin{equation}
\label{eqn:fundamental:far:low:2}
\int_{y\in B(x_0,r)}\int_{x\in A(x_0,R)} \abs{\nabla^{m-s}_x \nabla^{m-q}_y \mat E^L(x,y)}^2\,dx\,dy \leq C r^{2q} R^{2s} \biggl(\frac{r}{R}\biggr)^\varepsilon
.\end{equation}
If $q=\pdmn/2$ then we instead have the bound
\begin{equation}
\label{eqn:fundamental:far:lowest:2}
\int_{y\in B(x_0,r)}\int_{x\in A(x_0,R)} \abs{\nabla^{m-s}_x \nabla^{m-q}_y \mat E^L(x,y)}^2\,dx\,dy \leq C(\delta)\, r^{2q} R^{2s} \biggl(\frac{R}{r}\biggr)^\delta
\end{equation}
for all $\delta>0$ and some constant $C(\delta)$ depending on~$\delta$.

We also have the symmetry property
\begin{equation}
\label{eqn:fundamental:symmetric:low}
\partial_x^\gamma\partial_y^\delta E^L_{j,k}(x,y) = \overline{\partial_x^\gamma\partial_y^\delta E^{L^*}_{k,j}(y,x)}
\end{equation}
as locally $L^2$ functions, for all multiindices $\gamma$, $\delta$ with $\abs{\gamma}=m-q$ and $\abs{\delta}=m-s$.

If in addition $q+s>0$, then for all $p$ with $1\leq p\leq 2$ and $p<\pdmn/(\dmn-(q+s))$, we have that
\begin{equation}
\label{eqn:fundamental:near}
\int_{B(x_0,r)}\int_{B(x_0,r)} \abs{\nabla^{m-s}_x \nabla^{m-q}_y \mat E^L(x,y)}^p\,dx\,dy \leq C(p) r^{2\pdmn+p(s+q-\pdmn)}.\end{equation}
for all $x_0\in\R^\dmn$ and all $r>0$.

Finally, there is some $\varepsilon>0$ such that if $2-\varepsilon<p<2+\varepsilon$ then $\nabla^m\vec\Pi^L$ extends to a bounded operator $L^p(\R^\dmn)\mapsto L^p(\R^\dmn)$. If $\gamma$ satisfies $m-\pdmn/p<\abs{\gamma}\leq m-1$ for some such~$p$, then 
\begin{equation}
\label{eqn:fundamental:low}
\partial_x^\gamma
\Pi^L_j\mydot F(x) 
	= \sum_{k=1}^N \sum_{\abs{\beta}=m} \int_{\R^\dmn} 	\partial_x^\gamma\partial_y^\beta E^L_{j,k}(x,y)\,F_{k,\beta}(y)\,dy
	\quad\text{for a.e.\ $x\in\R^\dmn$}
\end{equation}
for all $\mydot F\in L^p(\R^\dmn)$ that are also locally in $L^{P}(\R^\dmn)$, for some $P>\pdmn/(m-\abs{\gamma})$. In the case of $\abs{\alpha}=m$, we still have that
\begin{equation}
\label{eqn:fundamental:2}
\partial^\alpha\Pi^L_j\mydot F(x) 
	= \sum_{k=1}^N \sum_{\abs{\beta}=m} \int_{\R^\dmn} 	\partial_x^\alpha\partial_y^\beta E^L_{j,k}(x,y)\,F_{k,\beta}(y)\,dy
	\quad\text{for a.e.\ $x\notin\supp \mydot F$}
\end{equation}
for all $\mydot F\in L^2(\R^\dmn)$ whose support is not all of $\R^\dmn$.
\end{thm}
Validity of the condition \eqref{eqn:fundamental:low} requires that we normalize $\vec\Pi^L\mydot F$ by decay at infinity, as in Lemma~\ref{lem:GNS}.

Before proving Theorem~\ref{thm:fundamental:high:2} in the case $2m\leq \dmn$, we mention two important corollaries.

First, we have the following uniqueness result.
\begin{lem}
\label{lem:fundamental:unique}
Let $E^L_{j,k}$ be the fundamental solution given by Theorem~\ref{thm:fundamental:high:2}. Let $m-\pdmn/2\leq\abs{\gamma}\leq m$, let $\abs{\beta}=m$, and let $1\leq j\leq N$, $1\leq k\leq N$.
Let $U$ and $V$ be two bounded open sets with $\overline U\cap \overline V=\emptyset$. Suppose that for some $\widetilde E^L_{j,k,\gamma,\beta}\in L^2(U\times V)$, 
\begin{equation*}\partial^\gamma \Pi^L_j(\1_V F\,\mydot e_{k,\beta})(x)=\int_{V} \widetilde E^L_{j,k,\gamma,\beta}(x,y)\,F(y)\,dy
\quad
\text{as $L^2(U)$-functions}\end{equation*}
for all $\mydot F\in L^2(V)$.

Then $\widetilde E^L_{j,k,\gamma,\beta}(x,y)=\partial_x^\gamma \partial_y^\beta E^L_{j,k}(x,y)$ as $L^2(U\times V)$-functions.

In particular, if $E^L_{j,k}$ and $\widetilde E^L_{j,k}$ both satisfy the conditions of Theorem~\ref{thm:fundamental:high:2}, then 
\begin{equation*}\widetilde E^L_{j,k}(x,y)= E^L_{j,k}(x,y) + \sum_{\abs{\gamma}<m-\pdmn/2} f_\gamma(x)\,y^\gamma + g_\gamma(y)\,x^\gamma\end{equation*}
for some functions $f_\gamma$ and~$g_\gamma$.
\end{lem}

Second, recall that if $\vec \varphi\in \dot W^2_m(\R^\dmn)$, then $\vec\varphi=\vec\Pi^L(\mat A\nabla^m\vec\varphi)$ as $\dot W^2_m(\R^\dmn)$-functions. Thus, if $\mydot F=\mat A\nabla^m\vec\varphi$ and $\gamma$ satisfy the conditions of formula~\eqref{eqn:fundamental:low}, then
\begin{equation}
\partial^\gamma\varphi_j(x) = \sum_{k=1}^N \sum_{\ell=1}^N \sum_{\abs{\alpha}=\abs{\beta}=m} \int_{\R^\dmn} \partial_x^\gamma\partial_y^\alpha E^L_{j,k}(x,y)\,A^{\alpha\beta}_{k\ell}(y) \,\partial^\beta \varphi_\ell(y)\,dy
\end{equation}
for almost every $x\in\R^\dmn$.

\begin{proof}[Proof of Theorem~\ref{thm:fundamental:high:2}]

Let $L$ be an operator of order $2m$ for some $m\leq \pdmn/2$. Construct the operator $\widetilde L$ as follows. Let $M$ be large enough that $\widetilde m=m+2M>\pdmn/2$, and let $\widetilde L= \Delta^M L \, \Delta^M$. That is, if $\vec u\in \dot W^2_{\widetilde m}(\Omega)$, then 
\begin{equation*}\bigl\langle \vec \varphi, \widetilde L \vec u\bigr\rangle_\Omega = \bigl\langle \Delta^M \vec\varphi, L\Delta^M\vec u\bigr\rangle_\Omega
\quad\text{for all smooth $\vec\varphi$ supported in~$\Omega$}
.\end{equation*}
Then $\widetilde L$ is a bounded and elliptic operator of order $2\widetilde m$, and so a fundamental solution $E^{\widetilde L}_{j,k}$ exists.

There exist constants $a_\zeta$ such that $\Delta^M \varphi = \sum_{\abs{\zeta}=2M} a_\zeta \partial^\zeta \varphi$ for all smooth functions~$\varphi$. Let 
\begin{equation*}
E^{ L}_{j,k}(x,y)
=
	\sum_{\abs{\zeta}=2M} \sum_{\abs{\xi}=2M} a_\zeta \, a_\xi\, \partial_x^{\zeta}\partial_y^{\xi} E^{\widetilde L}_{j,k}(x,y)
.\end{equation*}

We claim that $E^L_{j,k}$ satisfies the conditions of Theorem~\ref{thm:fundamental:high:2}.

First, notice that the symmetry formula~\eqref{eqn:fundamental:symmetric:low} and the bounds \eqref{eqn:fundamental:far:low:2}, \eqref{eqn:fundamental:far:lowest:2} and~\eqref{eqn:fundamental:near} follow immediately from the corresponding formulas for~$E^{\widetilde L}$.

We are left with formulas~\eqref{eqn:fundamental:low} and~\eqref{eqn:fundamental:2}; that is, we must now show that $\partial_x^\gamma\partial_y^\beta E^{ L}_{j,k}(x,y)$ is the kernel of the Newton potential. Choose some bounded, compactly supported function $\mydot F$ and some multiindex $\gamma$ with $m-\pdmn/2\leq\abs{\gamma}\leq m$, and let 
\begin{equation*}\widetilde F_{k,\widetilde \beta} = \sum_{\abs{\xi}=2M,\>\xi<\widetilde\beta} a_\xi \, F_{k,\widetilde\beta-\xi},\qquad\text{for all }\abs{\widetilde \beta}=\widetilde m.\end{equation*}
Let
\begin{equation*}v_j(x)=\sum_{k=1}^N \sum_{\abs{\beta}=m}\int_{\R^\dmn} \partial_x^\gamma\partial_y^\beta E^{ L}_{j,k}(x,y)
\,F_{k,\beta}(y)\,dy.\end{equation*}
We have that 
\begin{align*}
v_j(x)
&=
	\sum_{\abs{\zeta}=2M} a_\zeta  \sum_{k=1}^N 
	\sum_{\abs{\beta}=m}
	\int_{\R^\dmn} \, \partial_x^{\gamma+\zeta}\partial_y^{\beta+\xi} E^{\widetilde L}_{j,k}(x,y)
	\sum_{\abs{\xi}=2M}a_\xi\, F_{k,\beta}(y)\,dy
\\&=
	\sum_{\abs{\zeta}=2M} a_\zeta  \sum_{k=1}^N 
	\sum_{\abs{\widetilde\beta}=\widetilde m}
	\int_{\R^\dmn} \, \partial_x^{\gamma+\zeta}\partial_y^{\widetilde\beta} E^{\widetilde L}_{j,k}(x,y)
	\widetilde F_{k,\widetilde \beta}(y)\,dy
.\end{align*}
Formulas~\eqref{eqn:fundamental:low} and~\eqref{eqn:fundamental:2} are valid for $\mat E^{\widetilde L}$; thus we have that 
\begin{align*}
v_j(x)
&=
	\sum_{\abs{\zeta}=2M} a_\zeta  \partial_x^{\gamma+\zeta}\vec\Pi^{\widetilde L}_j
	\mydot{\widetilde F}(x)
=
	\partial_x^{\gamma}\Delta^M\vec\Pi^{\widetilde L}_j
	\mydot{\widetilde F}(x)
.\end{align*}
Thus, it suffices to show that $\Delta^M\vec\Pi^{\widetilde L} \mydot{\widetilde F}=\vec\Pi^L\mydot F$.

Choose some $\vec\varphi\in \dot W^2_m(\R^\dmn)$; then there is some $\widetilde\varphi\in \dot W^2_{\widetilde m}(\R^\dmn)$ with $\vec\varphi=\Delta^M\widetilde\varphi$. Then
\begin{align*}
\bigl\langle \nabla^m\vec\varphi, \mat A\nabla^m (\Delta^M\vec\Pi^{\widetilde L} \mydot{\widetilde F})\bigr\rangle_{\R^\dmn}
&=
	\bigl\langle \vec\varphi, L (\Delta^M\vec\Pi^{\widetilde L} \mydot{\widetilde F})\bigr\rangle_{\R^\dmn}
=
	\bigl\langle \Delta^M\widetilde\varphi, L (\Delta^M\vec\Pi^{\widetilde L} \mydot{\widetilde F})\bigr\rangle_{\R^\dmn}
.\end{align*}
But by definition of $\widetilde L$,
\begin{align*}
	\bigl\langle \Delta^M\widetilde\varphi, L (\Delta^M\vec\Pi^{\widetilde L} \mydot{\widetilde F})\bigr\rangle_{\R^\dmn}
&=
	\bigl\langle \widetilde\varphi, \widetilde L (\vec\Pi^{\widetilde L} \mydot{\widetilde F})\bigr\rangle_{\R^\dmn}
\end{align*}
and by definition of $\vec\Pi^{\widetilde L}$,
\begin{align*}
	\bigl\langle \widetilde\varphi, \widetilde L (\vec\Pi^{\widetilde L} \mydot{\widetilde F})\bigr\rangle_{\R^\dmn}
&=
	\bigl\langle \nabla^{\widetilde m} \widetilde\varphi, \mydot{\widetilde F}\bigr\rangle_{\R^\dmn}
.\end{align*}
Writing out the sums in the inner product and using the definition of $\widetilde F$, we see that
\begin{align*}
	\bigl\langle \nabla^{\widetilde m}\widetilde\varphi,  \mydot{\widetilde F}\bigr\rangle_{\R^\dmn}
&=
	\sum_{k=1}^N
	\sum_{\abs{\widetilde \beta}=\widetilde m}
	\bigl\langle \partial^{\widetilde \beta}\widetilde\varphi_k,  {\widetilde F}_{k,\widetilde\beta}\bigr\rangle_{\R^\dmn}
\\&=
	\sum_{k=1}^N
	\sum_{\abs{\widetilde \beta}=\widetilde m}
	\sum_{\abs{\delta}=2M,\>\delta<\widetilde\beta} 
	\bigl\langle \partial^{\widetilde \beta}\widetilde\varphi_k,  a_\delta \, F_{k,\widetilde\beta-\delta}\bigr\rangle_{\R^\dmn}
.\end{align*}
Interchanging the order of summation, we see that
\begin{align*}
	\bigl\langle \nabla^{\widetilde m}\widetilde\varphi,  \mydot{\widetilde F}\bigr\rangle_{\R^\dmn}
&=
	\sum_{k=1}^N
	\sum_{\abs{ \beta}= m}
	\sum_{\abs{\delta}=2M} 
	\bigl\langle  a_\delta \partial^{\delta+ \beta}\widetilde\varphi_k,  F_{k,\beta}\bigr\rangle_{\R^\dmn}
\end{align*}
and recalling the definitions of $a_\delta$ and $\widetilde\varphi$, we see that
\begin{align*}
	\bigl\langle \nabla^{\widetilde m}\widetilde\varphi,  \mydot{\widetilde F}\bigr\rangle_{\R^\dmn}
&=
	\sum_{k=1}^N
	\sum_{\abs{ \beta}= m}
	\bigl\langle \partial^{\beta}\Delta^M \widetilde\varphi_k,  F_{k,\beta}\bigr\rangle_{\R^\dmn}
\\&=
	\sum_{k=1}^N
	\sum_{\abs{ \beta}= m}
	\bigl\langle \partial^{\beta}\varphi_k,F_{k,\beta}\bigr\rangle_{\R^\dmn}
=
	\bigl\langle\nabla^m\vec\varphi,\mydot F\bigr\rangle_{\R^\dmn}
.\end{align*}
Thus,
\begin{align*}
\bigl\langle \nabla^m\vec\varphi, \mat A\nabla^m (\Delta^M\vec\Pi^{\widetilde L} \mydot{\widetilde F})\bigr\rangle_{\R^\dmn}
&=
	\bigl\langle\nabla^m\vec\varphi,\mydot F\bigr\rangle_{\R^\dmn}
.\end{align*}
By uniqueness of $\vec\Pi^L\mydot F$, this implies that  $\Delta^M\vec\Pi^{\widetilde L} \mydot{\widetilde F}=\vec\Pi^L\mydot F$, as desired.
\end{proof}





\section*{Acknowledgements}

The author would like to thank Steve Hofmann and Svitlana Mayboroda for many useful conversations concerning topics of interest in the theory of higher-order elliptic equations, and would also like to thank the American Institute of Mathematics for hosting the SQuaRE workshop on ``Singular integral operators and solvability of boundary problems for elliptic equations with rough coefficients,'' at which many of these discussions occurred.


\providecommand{\bysame}{\leavevmode\hbox to3em{\hrulefill}\thinspace}
\providecommand{\MR}{\relax\ifhmode\unskip\space\fi MR }
\providecommand{\MRhref}[2]{%
  \href{http://www.ams.org/mathscinet-getitem?mr=#1}{#2}
}
\providecommand{\href}[2]{#2}

\end{document}